\documentclass[12pt, reqno]{amsart}
\usepackage{amsmath,amssymb,mathtools}
\usepackage{amsmath,amssymb,mathtools,amsthm,
	textcomp,mathscinet}
\usepackage{amsfonts,graphicx,enumerate,subcaption,placeins}
\usepackage{times}
\usepackage{chngcntr}

\usepackage{amsfonts, amstext, amsthm, booktabs, dcolumn}

\usepackage[mathscr]{eucal}
\usepackage{pst-plot,pstricks-add,color}

\usepackage{amsmath,amssymb,mathtools,amsthm,textcomp,
	pst-plot,pstricks-add,mathscinet}

\usepackage{amsfonts,graphicx,enumerate}
\usepackage{color,tikz}

\usepackage[mathscr]{eucal}

\usepackage[american]{babel}
\newcommand \numberthis{\addtocounter{equation}{1}\tag{\theequation}}

\newcommand{\ncom}{\newcommand}

\ncom{\vone}{\vskip 2ex}
\ncom{\vtwo}{\vskip 4ex}

\ncom{\ba}{\begin{align}}
\ncom{\ea}{\end{align}}
\ncom{\bans}{\begin{align*}}
\ncom{\eans}{\end{align*}}
\ncom{\beq}{\begin{equation}}
\ncom{\eeq}{\end{equation}}
\ncom{\beqs}{\begin{equation*}}
\ncom{\eeqs}{\end{equation*}}

\ncom{\st}{\stackrel{\mathscr{L}}{=}}
\ncom{\isn}{\sum_{i=1}^{n}}
\ncom{\jsn}{\sum_{j=1}^{n}}
\ncom{\isy}{\sum_{i=1}^{\infty}}
\ncom{\jsy}{\sum_{j=1}^{\infty}}

\ncom{\eps}{\varepsilon}
\ncom{\sgm}{\sigma}
\ncom{\sgms}{\sigma^2}
\ncom{\Sgm}{\Sigma}
\ncom{\psgm}{\sigma^{\prime}}

\ncom{\nin}{\noindent}
\ncom{\noi}{\noindent}
\ncom{\non}{\nonumber}

\ncom{\al}{\alpha}
\ncom{\bt}{\beta}
\ncom{\Th}{\theta}
\ncom{\e}{\eta}
\ncom{\Ch}{\chi^2}
\ncom{\ga}{\gamma}
\ncom{\Ga}{\Gamma}
\ncom{\dt}{\delta}
\ncom{\Dt}{\Delta}
\ncom{\lmd}{\lambda}
\ncom{\Lmd}{\Lambda}
\ncom{\omg}{\omega}
\ncom{\Omg}{\Omega}

\ncom{\bga}{\boldsymbol{\gamma}}
\ncom{\blmd}{\boldsymbol{\lambda}}
\ncom{\bLmd}{\boldsymbol{\Lambda}}
\ncom{\bDt}{\boldsymbol{\Delta}}
\ncom{\bSgm}{\boldsymbol{\Sigma}}
\ncom{\bmu}{\boldsymbol{\mu}}
\ncom{\bTh}{\boldsymbol{\theta}}
\ncom{\beps}{\boldsymbol{\varepsilon}}
\ncom{\bPsi}{\boldsymbol{\Psi}}

\def \C{{{\rm I{\!\!\!}\rm C}}}

\def\N{{\mathbb N}}

\def\C{{\mathbb C}}

\ncom{\V}{\operatorname{Var}}
\ncom{\CV}{\operatorname{Cov}}
\ncom{\CR}{\operatorname{Corr}}
\def\E{{\mathbb E}}

\def\Z{{\mathbb Z}}

\def\rec#1{\frac{1}{#1}}

\numberwithin{equation}{section}
\def\theequation{\thesection.\arabic{equation}}
\newtheorem{theorem}{\bf Theorem}[section]
\newtheorem{remark}{\bf Remark}[section]
\newtheorem{remarks}{\bf Remarks}[section]
\newtheorem{proposition}{Proposition}[section]
\newtheorem{lemma}{Lemma}[section]
\newtheorem{corollary}{Corollary}[section]

\newtheorem{definition}{Definition}[section]

\newtheorem{defn}{Definition}[section]

\begin{document}

\hfill \today
\vtwo
\title{A Unified Approach to Compound Poisson Process and its Time  fractional Versions}
\author{Palaniappan Vellaisamy and Tomoyuki Ichiba} 
\address{Department of Statistics and Applied Probability, UC Santa Barbara,
	Santa Barbara, CA, 93106, USA.}
\address{Department of Statistics and Applied Probability, UC Santa Barbara,
Santa Barbara, CA, 93106, USA.}
\email{pvellais@ucsb.edu, ichiba@pstat.ucsb.edu}
\date{}

\begin{abstract} We consider a weighted sum of a series of independent Poisson random variables and show that it results in a  new compound Poisson distribution which includes the Poisson distribution and Poisson distribution of order $k$ and Poisson distribution of order infinity.  An explicit representation for its distribution is obtained in terms of Bell polynomials. We then extend it to a  compound Poisson process and time fractional compound Poisson process (TFCPP). It is shown that the one-dimensional distributions of the TFCPP  exhibit over-dispersion property, are not infinitely divisible and possess the long-range dependence property. Also, their moments and factorial moments are derived. Finally, the fractional differential equation associated with the TFCPP is also obtained. Some possible applications to insurance are pointed out.

\end{abstract}

\maketitle

\section{Introduction}
In recent decades, the classical Poisson process, negative binomial process
and gamma process have been generalized to various forms of their fractional versions such as
fractional Poisson process (FPP), fractional negative binomial processes (FNBP) and fractional gamma process (FGP);
see, for example, Laskin \cite{Lask03}, Begin and Orsingher \cite{BegOrs09}, 
Begin \cite{Beg13}, Mainardi {\it et al.} \cite{MGS04}, Begin and Macci  \cite{BegMac14}, Meerschaert {\it et al.} \cite{MNV09}, Meerschaert {\it et al.} \cite{MNV11}, Vellaisamy and Maheswari \cite{PVAM18}, Khandakar {\it et al.}
\cite{KKV25} and Kataria and Khandakar \cite{KKMK22}. These time and space fractional versions have heavy-tailed distributions, non-exponential waiting times and long-range dependence properties; see Biard and Saussereau \cite{BDSS14}, and Maheshwari and Vellaisamy \cite{MV16}.  These characteristics make these processes more suitable than the classical Levy processes (see Applebaum \cite{AB09}) for modeling various phenomena that arise in many disciplines such as finance, hydrology, atmospheric science, etc. (see Laskin \cite{Lask09}).
The time fractional Poisson process which initially derived from certain fractional differential equations (see Laskin \cite{Lask03}, Begin and Orsingher \cite{BegOrs09}) can also be viewed as a Poisson process subordinated to inverse stable subordinator (see Meerschaert {\it et al.} \cite{MNV11}).
 This approach initiated  the study of various subordinated processes leading to time-fractional
 and space fractional  versions; some references, among others, are Orsingher and Polito \cite{OrsPoli12}, Maheswari and Vellaisamy \cite{MV19}, Leonenko {\it et al.} \cite{LMS14}  and Begin and Vellaisamy \cite{BegVel18}.
 
 \vspace{.3cm} \noi In this paper, we look at the sum of a series of independent weighted Poisson random variables (rvs), which leads to a new compound Poisson distribution (CPD). By suitably choosing the associated sequence of parameters, we show that the CPD includes the Poisson distribution, the Poisson distribution of order $k$ and infinity, and so its study leads to a unified approach. Replacing Poisson rvs by Poisson processes, we extend it to a new compound Poisson process (CPP). Later, as a natural extension, we consider its time fractional version and call it TFCPP.
 
 \vspace{.3cm} \noi In Section 2, we introduce the notations and required preliminary results
 and the CPD is introduced and studied in Section 3. In Section 4, the associated CPP
 is studied and its time fractional version TFCPP  is investigated in Section 5.
 In particular, the one-dimensional distributions of the TFCPP have the over-dispersion
 property, are not infinitely divisible, and possess the long-range dependence property.
 The moments and the factorial moments of the TFCPP  are derived. Finally, its factional differential
 equation satisfied by the TFCPP is also derived.

\section{Preliminaries}\label{secprelim}
\noindent In this section, we introduce the notation and results that will be used later.
We start with some special functions that will be required later.

\subsection{Some special functions}

\begin{defn}\label{p} \em  (i): The one parameter Mittag-Leffler function $M_{\beta}(z)$ is defined as (see \cite{erde3})
	\begin{equation}\label{mlf22}
		M_{\alpha}(z)=\sum\limits_{k=0}^{\infty}\frac{z^{k}}{\Gamma(\alpha k + 1)},\,\,\,\alpha, z\in \C \text{ and Re}(\alpha)>0.
	\end{equation}
(ii):	For $z \in \mathbb{C}$, the two parameter Mittag-Leffler  function is defined as 
	\begin{equation}\label{mlf22a}
		M_{\alpha,\beta}(z)=\sum_{k=0}^{\infty} \frac{z^k}{\Gamma(\alpha k + \beta)}, ~~~~\text{$\alpha > 0, \beta > 0.$}
	\end{equation}
\end{defn}
\noindent
When $\beta = 1, M_{\alpha,1}$ reduces to the one-parameter Mittag-Leffler function
$ M_{\alpha}(z).$ \\
(iii): The generalized Mittag-Leffler function (Prabhakar \cite{PB71}) is defined as
\begin{equation}\label{k}
	M_{\alpha,\beta}^\gamma (z)=  \frac{1}{\Ga(\ga)}\sum_{k=0}^{\infty} \frac{\Gamma(\gamma+k)}{k!\Gamma(\alpha k + \beta)}{z^k},~~~\text{$\alpha, \beta, \gamma, z \in \mathbb{C}$}
\end{equation}
where $Re(\alpha), Re(\beta), Re(\gamma) > 0$. Note that $	M_{\alpha,\beta}^\gamma (0)=
1/\Ga(\bt),$ using the convention $0^0=1.$

Let $M_{\alpha,\beta}^{(n)}(x)$ denote the $n$-th derivative of the two-parameter Mittag-Leffler function
defined in \eqref{mlf22a}. Then (see Kataria and Vellaisamy \cite{KV19}, p.~117)
\begin{equation}\label{2.5}
	M_{\alpha,\beta}^{(n)}(x)=n! M_{\alpha,n\alpha+\beta}^{n+1}(x),\quad n\ge 0.
\end{equation}
Also, the $n$-th derivative of $ M_{\alpha}(z)$ satisfies
\begin{equation}\label{2.5a}
\alpha  M_{\alpha}^{(n)}(-\lambda t^{\alpha})= M_{\alpha,\alpha }^{(n-1)}(-\lambda t^{\alpha}), \quad n \geq 1.
\end{equation}

\noindent The Mittag-Leffler $ ML (\alpha,  \lambda) $ distribution is introduced and studied by Pillai \cite{Pil90}. Its distribution function is given by
\begin{equation}\label{mld1}
	F(t \mid \alpha, \lambda)= 1- M_{\alpha}(-\lambda t^{\alpha})=\sum_{k=0}^{\infty} (-1)^{k+1} \frac{(\lambda t^\alpha)^k}{\Gamma(\alpha k+1)}, ~ t>0,
\end{equation}
where $ 0 < \alpha \leq 1$ and $\lambda >0$ are the parameters.

\noi Let $M_{\alpha}^{'}(x)=M_{\alpha}^{(1)}(x)$ denote the first derivative. Then the density of $ML (\alpha,  \lambda) $ is given by
\begin{align}\label{mld2}
	f(t \mid \alpha, \lmd)=& \alpha \lambda  t^{\alpha-1} M_{\alpha}^{'}(-\lambda t^{\alpha})	\nonumber\\	
	=& \lambda  t^{\alpha-1} M_{\alpha,\alpha }(-\lambda t^{\alpha}),		
\end{align}
using  \eqref{2.5a}. 

\noi Let a random variable (rv) $X \sim ML ( \alpha,  \lambda)$. The its Laplace transform  is
$$ \mathbb E(e^{-sX})= \frac{\lambda}{\lambda + s^ \alpha}.$$
If  $T=  X_1 + \cdots + X_n$, where $X_i$'s are independent and identically distributed (i.i.d.) 
$ML (\alpha, \lambda)$ random variables (rvs), then the density of $T$ (see Kataria and Vellaisamy \cite{KV19}) is
\begin{align}\label{mld3}
	f(t\mid\alpha, \lambda)=& \, \frac{\lambda ^n }{(n-1)!}
	t^{\alpha n -1} M_{\alpha,\alpha }^{(n-1)}(-\lambda t^{\alpha}) \non	\\	
    =& \, \frac{\alpha\lambda ^n }{(n-1)!}
    t^{\alpha n -1} M_{\alpha }^{(n)}(-\lambda t^{\alpha}),
\end{align} 
using \eqref{2.5a}.

\begin{definition}  \em
	\noindent  (i): The Wright function $W _{\alpha, ~\beta }(z)$ is defined for $\alpha > -1$ and $\beta \in \mathbb{C}$, as
	\begin{align} \label{eq22a}
		W _{\alpha, \beta }(z)= {\sum_{n=0}^{\infty}
			\frac{z^n }{ n! \Gamma(\alpha n + \beta )}},
	\end{align}
	which converges in the whole complex plane.
	
	\noi (ii):  A particular case of the Wright function, called {the $M$-Wright function} $W_{\beta}(z)$ is defined as
	\begin{align} \label{eq22}
		W_{\beta}(z) = &  \,  W _{-\beta, 1-\beta }(-z)
		={\sum_{n=0}^{\infty}
			\frac{(-z)^n }{ n! \Gamma(-\beta n + (1-\beta ))}} \\
		= & \, {\rec{\pi} \sum_{n=1}^{\infty}\frac{(-z)^{n-1} }{(n-1)!}
			\Gamma(\beta n)  \sin (\pi\beta n)}, \nonumber
	\end{align}
	which converges for  $z\in \mathbb{C}$ and $0<\beta<1$. 
	
\end{definition}

\noi  The following results are well known (see  Kilbas {\t et al.} \cite{KST06}, Kataria and Vellaisamy \cite{KV18}).
\begin{definition}  \em (i): 
	Let $D_{\beta}(t)$ be the $\beta$-stable subordinator. Then the density of $D_{\beta}(t)$ is
	\begin{equation}\label{stab-den}
		g_{_{\beta}}(x;t)=\beta tx^{-(\beta+1)}W_{\beta}(tx^{-\beta}), ~~~x>0.
	\end{equation}

(ii):	Let $E_{\beta}(t)$ be the inverse $\beta$-stable subordinator. Then the density of $E_{\beta}(t)$ is 
	\begin{equation}\label{instab-den}
		h_{\beta}(x;t)=t^{-\beta}W_{\beta}(t^{-\beta}x),~~~x>0.
	\end{equation}
\end{definition}

\subsection{Some fractional derivatives} 
Let $AC[a,b]$ be the space of functions $f$ which are absolutely continuous on $[a,b]$ and 
\begin{equation*}
	AC^{n}[a,b]=\Big\{f:[a,b]\rightarrow \mathbb{R}; f^{(n-1)}(t)\in AC[a,b]\Big\},
\end{equation*}
where $AC^{1}[a,b]=AC[a,b]$.

\vone
\noi Henceforth, let $ \Z_{+}= \{0, 1, \ldots \}$ and $ \N = \{1, 2, \ldots \}$ be the set
of nonnegative integers and the set of positive integers respectively.  

\begin{definition} \em
	Let $n\in \N$,  $\beta \geq 0$ and $f \in AC^{n}[0,T]$. Then \\ 
	(i): The (left-hand) Riemann-Liouville (R-L) fractional derivative  $\mathbb{D}_t^{\beta} f $ of $f$ (see \cite[Lemma 2.2]{KST06}) is defined by  (with $\mathbb{D}^{0} f=f$), for $t \in [0, T]$,
	\begin{equation}\label{rld}
		\mathbb{D}_t^{\beta}f(t):= \begin{cases} 
			\hfill \dfrac{1}{\Gamma(n-\beta)}\dfrac{d^{n}}{dt^{n}}\displaystyle\int_{0}^{t}\dfrac{f(s)}{(t-s)^{\beta-n+1}}ds, \hfill    &  \text{if}~~ n-1<\beta < n, \\ & \\
				f^{(n)}(t), \quad  \text{if}~~ \beta= n. &
		\end{cases}
	\end{equation}
(ii): The (left-hand) Caputo fractional derivative $\partial^{\beta} f$ of $f$ (see  Kilbas {\it et al.} \cite[Theorem 2.1]{KST06}) is defined by (with $\partial^{0}f=f$), for $t \in [0, T]$,
 \begin{equation}\label{cd}
 	\partial_t^{\beta}f(t):= \begin{cases} 
 		\hfill \dfrac{1}{\Gamma(n-\beta)}\displaystyle\int_{0}^{t}\dfrac{f^{(n)}(s)}{(t-s)^{\beta-n+1}}ds, \hfill    & \text{if}~~ n-1<\beta < n \\ &\\
 		f^{(n)}(t), \quad \text{if}~~ \beta= n.&
 	\end{cases}
\end{equation}
\end{definition}

\noindent The relation between the R-L fractional derivative and the Caputo fractional derivative is (see Kilbas {\it et al.} \cite[eq. (2.4.6)]{KST06}), 
\begin{equation*}
	\mathbb{D}_t^{\beta}f(t)	= \partial_t^{\beta}f(t)    +\displaystyle\sum\limits_{k=0}^{n-1}\frac{t^{k-\beta}}{\Gamma(k-\beta+1)}f^{(k)}(0^{+}),~n-1<\beta \le n,
\end{equation*}
where $f^{(k)}(0^+):=\lim_{t\rightarrow 0^+} \frac{d^{k}}{dt^{k}}f(t)$.

\vone
\subsection{Poisson Distribution of Order $k$}
We denote hereon the Poisson distribution with parameter $\lmd$ by  Poi($\lmd$).   Let $\{N_{j}\}, j \geq 1,$ where $N_{j}\sim\text{Poi}(\lmd_{j})$, be a sequence of  independent Poisson random
variables   with probability distribution 
\begin{equation*}
\mathbb{P}\{N_{j}=n\}=\frac{e^{-\lmd_{j}}\lmd_{j}^{n}}{n!},~n\in \mathbb{Z_{+}}, ~ j \in \N.
\end{equation*}%
 It is well known that the Poisson family is stable under convolution,
that is, the sum
\begin{equation*}
S_{k}=\sum_{j=1}^{k}N_{j}\sim \text{Poi}(\lmd_{1}+\cdots +\lmd_{k}), \quad k \geq 1,
\end{equation*}%
the Poisson distribution with mean $\lmd_{1}+\cdots +\lmd_{k}$, for all $k \ge 1.$

\noi For $ k \geq 1$,  consider now the
random variable $W_{k}=N_{1}+2N_{2}+\cdots + k N_{k}$, a weighted sum of Poisson rvs. Its
distribution is given by, for $n\in \mathbb{Z_{+}}$, 
\begin{align}
\mathbb{P}\{W_{k}=n\}& =\sum_{\big\{\sum\limits_{j=1}^{k}jx_{j}=n\big\}}\mathbb{P}\{N_{1}=x_{1},\cdots ,N_{k}=x_{k}\}  \notag \\
& =\sum_{\big\{\sum\limits_{j=1}^{k}jx_{j}=n\big\}}\prod_{j=1}^{k}\mathbb{P}\{N_{j}=x_{j}\}  \notag
\\
& =e^{-(\lmd_{1}+\cdots +\lmd_{k})}\sum_{\big\{ \sum\limits_{j=1}^{k}jx_{j}=n \big\}
}\Big( \frac{\lmd_{1}^{x_{1}}\cdots \lmd_{k}^{x_{k}}}{x_{1}!\ldots x_{k}!}\Big),  \label{eqn214}
\end{align}
where we have used the  independence of the $X_i$'s in the second line above. The probability distribution given
in \eqref{eqn214} is called the Poisson distribution of order $k$ (see Philippou {\it et al.} \cite
{Philip88}) and is denoted by Poi ($\lmd_{1},\ldots , \lmd_{k})$. Thus, it follows
that 
\begin{equation} \label{n216}
W_{k}=\sum_{j=1}^{k}jN_{j}\sim \text{Poi}(\lmd_{1},\ldots , \lmd_{k}), \quad k \geq 1,
\end{equation}
 Recently, it has been shown (see \cite{PVFV20},  \cite{PVVK22}) that the Poisson distribution of order $k$ 
 plays a crucial role in the probabilistic analysis of both the classic
 and the non-classic  Adomian polynomials and especially in the derivation of
 recurrence relations.

\subsection{Bell Polynomials} The following definitions and results of Bell polynomials will be required later.

\noindent
\subsubsection{Ordinary Bell polynomials.} 
Let $c_j$'s denote nonnegative integers. For $ 1 \leq k \leq n$, define
\begin{align*}
	\Lambda_{k,n}= \Big\{(c_1,c_2,\ldots,c_n):c_{1}+\cdots+c_{n}=k;~c_{1}+2c_{2}+\cdots+nc_{n}=n \Big \}
\end{align*}
and
\begin{align*}
	\Lambda_{k,n}^{*}= \Big\{(c_1,c_2,\ldots,c_{n-k+1}):c_{1}+\cdots+c_{n-k+1}=k;~c_{1}+2c_{2}+\cdots+(n-k+1)c_{n-k+1}=n\Big \}.
\end{align*}

\nin Note that, for fixed $ 1 \leq k \leq n$,
\begin{align} \label{eqn217b}
   (c_1, \ldots, c_{n-k+1}) \in \Lambda_{k,n}^{*} \iff (c_1, \ldots, c_{n-k+1}, 0, \ldots, 0) \in \Lambda_{k,n}.
\end{align}

\nin Also,  for $n\geq1$, let
\begin{align} \label{eqn217}
	\Lambda_{n}=\Big\{(c_1,c_2,\ldots,c_n):c_{1}+2c_{2}+\cdots+nc_{n}=n\Big \}.
\end{align}
\noi The  ordinary partial Bell polynomials are defined by
\begin{align} \label{bp2}
	\hat{B}_{n,k}(u_{1},\dots,u_{n-k+1}) 
	=\sum_{\Lambda_{k,n}^*}{\frac{k!}{c_{1}!\cdots c_{n-k+1}!}}{{u_{1}}}^{c_{1}}\cdots{{u_{n-k+1}}}^{c_{n-k+1}}
\end{align}
or equivalently
\begin{align} \label{219}
	\hat{B}_{n,k}(u_{1},\dots,u_{n})= \sum_{\Lambda_{k,n}}{\frac{k!}{c_{1}!\dots c_{n}!}}{{u_{1}}}^{c_{1}}\dots{{u_{n}}}^{c_{n}},
\end{align}
in view of \eqref{eqn217b}.

\noi The  ordinary  Bell polynomials are defined by
\begin{align} \label{bp4}
  	\hat{B}_{n}(u_{1},\dots,u_{n})= \sum_{k=1}^{n}\hat{B}_{n,k}(u_{1},\dots,u_{n}),\quad n\geq 1. 
\end{align}

\noi The following results are well known (see Comtet \cite{CT74}, pp. 133-137):
\begin{align*}
	\exp \Big(x \sum_{j=1}^{\infty} \frac{u_{j} t^{j}}{j!}\Big)=  1+\sum_{n=1}^{\infty} \frac{t^{n}}{n!}\Big( \sum_{k=1}^{n} \hat{B}_{n, k}\Big(u_{1}, u_{2}, \ldots, u_{n-k+1}\Big)x^{k}\Big)
\end{align*}
and also for the series expansion of the $k$-th power, for $ k \geq 1,$ 
\begin{equation}\label{221}
	\Big(\sum_{j=1}^{\infty} u_{j} t^{j}\Big)^{k}=\sum_{n=k}^{\infty} \hat{B}_{n, k}\Big(u_{1}, u_{2}, \ldots, u_{n-k+1}\Big) t^{n}.
\end{equation}

\noi \subsubsection{Partial  exponential Bell polynomials.}
The partial or incomplete exponential Bell polynomials are a triangular array of polynomials given by
\begin{align*}
	B_{n,k}\Big(u_{1},u_{2},\ldots,u_{n-k+1}\Big)= &\sum_{\Lambda_{k, n}^{*}}\frac{n!}{c_{1}!c_{2}!\cdots c_{n-k+1}!}\Big(\frac{u_{1}}{1!}\Big)^{c_{1}}\Big(\frac{u_{2}}{2!}\Big)^{c_{2}} \cdots\Big(\frac{u_{n-k+1}}{(n-k+1)!}\Big)^{c_{n-k+1}} \\
\end{align*}
or compactly
\begin{align*}
B_{n,k}\Big(u_{1},u_{2},\ldots,u_{n}\Big)=\sum_{\Lambda_{k, n}}\binom{n}{c_1,c_2,\ldots,c_n}\Big(\frac{u_{1}}{1 !}\Big)^{c_{1}}\Big(\frac{u_{2}}{2!}\Big)^{c_{2}} \cdots\Big(\frac{u_{n}}{n!}\Big)^{c_{n}},
\end{align*}
where $B_{0,0}=1$, $B_{0,k}=0~,k\geq 1$. 
Here, $\binom{n}{c_1,c_2,\ldots,c_n} $ denotes the multinomial coefficient.

\noi  Let $\Lmd_n$ be defined as in  \eqref{eqn217}. Then the sum
\begin{align}\label{222}
	B_{n}\Big(u_{1}, \ldots, u_{n}\Big)= &\sum_{k=1}^{n} B_{n, k}\Big(u_{1}, u_{2}, \ldots, u_{n}\Big) \nonumber \\
      = & \sum_{\Lambda_{n}} \binom{n}{c_1, c_2, \ldots, c_n}
     \Big(\frac{u_{1}}{1 !}\Big)^{c_{1}}\Big(\frac{u_{2}}{2 !}\Big)^{c_{2}} \cdots\Big(\frac{u_{n}}{n!}\Big)^{c_{n}},
\end{align}
with $B_0=1$, is called {$n$-th complete exponential Bell polynomial}. 
 Henceforth,  Bell polynomials always refer to exponential Bell polynomials, unless
stated otherwise.

\noi The relation between ordinary Bell polynomials and exponential Bell polynomials is
\begin{equation}\label{eqn312}
	\hat{B}_{n,k}\Big(u_{1},u_{2},\ldots,u_{n}\Big)=\frac{k!}{n!}B_{n,k}\Big(1!u_{1},2!u_{2},\ldots,n!u_{n}\Big).
\end{equation}

\noindent
The following result from Johnson \cite[p.~220]{John02} is useful. Let $|A|$ denote the determinant of the matrix $A.$
\begin{lemma}  \em \label{det}
	If $n \geq 1$, then
	\begin{equation*}
		B_{n}\Big(u_{1}, u_{2}, \ldots, u_{n}\Big)=
		\begin{vmatrix}
			u_1&\binom{n-1}{1}u_2&\binom{n-1}{2}u_3&\cdots&\binom{n-1}{n-2}u_{n-1}&u_n\\
			&&&&&\\
			-1&u_1&\binom{n-2}{1}u_2&\cdots&\binom{n-2}{n-3}u_{n-2}&u_{n-1}\\
			&&&&&\\
			0&-1&u_1&\cdots&\binom{n-3}{n-4}u_{n-3}&u_{n-2}\\
			\vdots&\vdots&\vdots&\vdots&\vdots&\vdots&\\
			0&0&0&\cdots&u_1&u_2\\
			0&0&0&\cdots&-1&u_1\\
		\end{vmatrix}.
	\end{equation*}
	\noi  Note in the above matrix, all the entries on the main subdiagonal are $-1$, and all entries below it are 0 .
\end{lemma}

\corollary  \em  Let $u_j=0$ for $j\ge 2$. Then 
by using Lemma \ref{det}
\begin{align*}
	B_n(u_1,0,0,\ldots,0)= &\begin{vmatrix}
		u_1&0&0&\cdots&0&0\\
		-1&u_1&0&\cdots&0&0\\
		0&-1&u_1&\cdots&0&0\\
		\vdots&\vdots&\vdots&\vdots&\vdots&\vdots&\\
		0&0&0&\cdots&u_1&0\\
		0&0&0&\cdots&-1&u_1
	\end{vmatrix} \\
	=& \, u_1^n.
\end{align*}

 \subsubsection{ A Probabilistic Formula.} In this subsection, we present a probabilistic approach to compute the Bell polynomials. Let $X_j \sim \text{Poi} (\lambda_j) $  be a sequence of independent Poisson rvs and $ u_j = j! \lmd_j,  j \geq 1$.  As before, let $S_n= X_1 + \cdots +  X_n$ and $W_n= X_1 + 2X_2+ \cdots +  nX_n$. 
Then by definition
\begin{align*}
	B_{n,k}(u_{1},\dots,u_{n}) &= \sum_{\Lambda_{k,n}}\binom{n}{c_1,\ldots, c_n} \Big( {\frac{u_{1} }{1! }} \Big)^{c_{1} }\dots
	\Big( {\frac{u_{n}}{n!}}\Big)^{c_{n}} \\
	&= n! \sum_{\Lambda_{k,n}} \dfrac{\lambda_1^{c_1}}{c_1!} \cdots \dfrac{\lambda_n^{c_n}}{c_n!} \\
	&= n! e^{\lambda_1 + \cdots + \lambda_n}\sum_{\Lambda_{k,n}}\dfrac{e^{-\lambda_1}\lambda_1^{c_1}  }{c_1!} \cdots \dfrac{e^{-\lambda_n}\lambda_n^{c_n}}{c_n!} \\
	&= \dfrac{n! \mathbb{P}\{S_n=k; W_n= n\}}{\mathbb{P}\{S_n=0\}}. \numberthis \label{eqn313}
\end{align*}
Using \eqref{222},  the complete Bell  polynomial is for $ n \geq 1$,
\begin{align*}
	B_n(u_1, \ldots, u_n)=& \sum_{k=1}^{n}	B_{n,k }( u _{1} , \dots, u _ {n} )\\
	= & \, n! \sum_{k=1}^{n} \dfrac{\mathbb{P}\{S_n=k; W_n= n\}}{\mathbb{P}\{S_n=0\}}\\
	=  & \, \frac{n!}{P\{S_n=0\} } \sum_{k=0}^{n}  P\{S_n=k; W_n= n \}\\
	=  & \, \frac{n! \mathbb{P}\{W_n= n \}}{\mathbb{P}\{S_n=0\} }, \numberthis \label{eqn314} 
\end{align*}
since $\mathbb{P}(S_n=0; W_n= n )=0$ for $ n \geq 1.$

\noi In other words,  
\begin{align} \label{eqn226}
	B_n(1!\lmd_1,2!\lmd_2,\ldots,n!\lmd_n)=\frac{n!\mathbb{P}\{W_n= n\}}{\mathbb{P}\{S_n=0\}}, 
\end{align}
where $X_j \sim \text{Poi} (\lambda_j)$.

\noi  For some additional details on the probabilistic connections to the Bell polynomials, see Kataria and Vellaisamy \cite{KVV22}.

\section{A Compound Poisson Distribution}
\noindent Let $\{N_j \}$ be a sequence of independent Poisson variables, where $\{N_j \} \sim \text{Poi} (\lambda_j), j \geq 1.$  Henceforth, we define $N_j \equiv 0 ~ a.s$ if $ \lmd_j=0,$, that is, 
a Poisson distribution with mean zero is defined to be the degenerate distribution at zero.
 Let  $ \blmd= \{ \lambda_j\}_{j \geq 1}$ be the sequence of
associated parameters such that $\dt=\sum\limits_{j=1}^{\infty}\lambda_{j} < \infty$ 
and $\sum\limits_{j=1}^{\infty}t^j\lambda_{j} < \infty,$ for   $ 0 < |t| < M,$  for some $M>0.$ We call the distribution of
\begin{equation} \label{neqn31}
G_P(\blmd)=\sum_{j=1}^{\infty}jN_{j},
\end{equation}
a generalized Poisson distribution (GPD). Note that $G_P(\blmd)= W_{\infty}$, see \eqref{n216}. Our generalization of the Poisson distribution is new and different from the ones considered in the literature; see, for instance, Consul and Jain \cite{CJ73} where  a two-parameter  generalization of the Poisson distribution is obtained as a limiting form of the generalized negative binomial distribution.

\begin{remarks} \em  (i) Let now $ \lmd_j =0 $ for $ j \geq 2.$ Then, obviously, $ G_P(\blmd) {=} N_1 \sim \text{Poi}(\lmd_1)$,
	the Poisson distribution with mean$\lmd_1.$
	
	\noi 	(ii) Let  $\lambda_j =0$, for $ j \geq k+1.$ Then, clearly, $G_P(\blmd){=} N_1+2N_2+ \ldots+ kN_k =W_k \sim \text{Poi} (\lmd_{1},\ldots, \lmd_{k})$, the Poisson distribution of order $k.$ 
\end{remarks}
\noi The study of the GPDs in \eqref{neqn31} provides a unified approach and has not been addressed in the literature. In view of the form of \eqref{neqn31}, the GPD may be viewed as the  Poisson distribution of order infinity.
First, we show that the series in the right-hand side of \eqref{neqn31} indeed follows a compound Poisson distribution. One could also apply Kolmogorov's three series theorem to check the convergence.

\vspace{.3cm}
\noi Let now $Y$ be a positive integer-valued random variable with the distribution
\begin{equation} \label{neqn32}
	\mathbb{P}\{Y=j\}= \frac{\lmd_j}{\dt}, ~j \in \N.
\end{equation}
where $\dt=\sum\limits_{j=1}^{\infty}\lambda_{j}.$
Observe that given a sequence $ \blmd$, satisfying the conditions given above, the distributions of $N\sim \text{Poi} ~(\delta)$ and $Y$ can be determined.

\begin{theorem} \em Let $G_P(\blmd)$ be the GPD defined in \eqref{neqn31} and $\{Y_j \}_{j \ge 1}$ be a sequence of i.i.d. positive-integer valued rvs with distribution defined in 
\eqref{neqn32}.	Also, let  $N\sim \text{Poi}(\dt)$ be a Poisson rv  with  mean $\dt$ and is independent of
the sequence $\{Y_j \}$. Then
\begin{equation}\label{neqn33}
\sum_{j=1}^{\infty}jN_{j}\stackrel{\mathscr{L}}{=} \sum_{j=1}^{N}Y_j.	
\end{equation}
\end{theorem}

\begin{proof} First note that the probability generating function (PGF) of $N$ is 
\begin{equation*}
	H_{N}(z)=\mathbb{E}\Big( z^{N}\Big) =e^{-\dt(1-z)}, ~ z\in (0, 1).
\end{equation*}
Consider now the random sum $T_N=\sum_{j=1}^{N}Y_j.$ Then  the PGF of $T_N$ is 
\begin{align}
	H_{T_{N}}(z) 
	&= H_{N} (H_{Y_1}(z)) \nonumber \\
	&= e^{ -\dt (1-H_{Y_1}(z)) }\nonumber \\  
	&=\exp \Big({\dt \Big(\sum_{j=1}^{\infty}z^{j}\frac{\lmd_{j}}{\dt}- 1\Big)} \Big)\nonumber\\
	&=\exp\Big(\sum_{j=1}^{\infty}z^{j}\lmd_{j}-\dt \Big) \label{eqn33a} \\
	&=\exp\Big(\sum_{j=1}^{\infty}z^{j}\lambda_{j}-\sum_{n=1}^{\infty}\lambda_{j}\Big)\nonumber\\
	&=\exp\Big(\sum_{j=1}^{\infty}\lambda_{j}(z^{j}-1)\Big). \numberthis \label{neqn33a}
\end{align}
Also, the PGF of $ \sum_{j=1}^{\infty}jN_{j}$ is
\begin{align}
	\mathbb{E}\Big( z^{N_{1}+2N_{2}+\cdots}\Big) \non   = &\prod_{j=1}^{\infty}\mathbb{E}\Big( z^{jN_{j}}\Big) \non \\
	=&\prod_{j=1}^{\infty}\mathbb{E}\Big( (z^{j})^{N_{j}}\Big) \non  \\
	=&\prod_{j=1}^{\infty}e^{-\lambda_{j}(1-z^{j})}  \non\\
	=&\exp\Big({\sum_{j=1}^{\infty}\lambda_{j}(z^{j}-1)}\Big), \label{neqn34}
\end{align}
which coincides with \eqref{neqn33a}. This proves the result.
\end{proof}

\noi The above result motivates the following alternative definition.
\begin{definition} \em  {\bf A Compound Poisson Distribution.}
	Let $\{N_j \}$ be a sequence of Poisson rvs with parameters $\lmd_j$ and $\dt=\sum\limits_{j=1}^{\infty}\lambda_{j}$. Let  $\{Y_j \}$ be a sequence of i.i.d. rvs with $\mathbb{P}(Y_{1}=j)={\lambda_{j}}/{\dt}$ for $j\ge 1$. We call the distribution
	$ \sum_{j=1}^{N}Y_j =T_N,$
	where $N \sim \text{Poi}(\dt)$ and is independent of $Y_j$, the compound Poisson 
	distribution (CPD) and denote it by $C_P(\blmd).$
\end{definition}

\begin{remarks} \em  (i) Let $ \lmd_j =0 $ for $ j \geq 2.$ Then, obviously,  $\dt= \lmd_1$ and $T_N \stackrel{\mathscr{L}}{=} N_1 \sim \text{Poi}(\lmd_1)$,
	the Poisson distribution.
	
	\noi 	(ii) Let  $\lambda_j =0$, for $ j \geq k+1.$ Then $\dt= \dt_k= \sum_{j=1}^{k} \lmd_j$ and $T_N \stackrel{\mathscr{L}}{=} N_1+2N_2+ \ldots+ kN_k= T_N^{(k)}$ {(say)} follows $\text{Poi} (\lmd_{1},\ldots, \lmd_{k})$, the Poisson distribution of order $k$ and has
	the PGF  
\begin{align} \label{36a}
\exp\Big({\sum_{j=1}^{k}-\lambda_{j}(1-z^{j})}\Big).
\end{align}

\nin  (iii) Note also that, in view of \eqref{neqn33},  $T_N^{(k)}= \sum_{j=1}^{M_k}Y_j$, where   $M_k \sim \text{Poi}(\dt_k)$, where $\dt_k= \sum_{j=1}^{k} \lmd_j$. Let $\phi_{Y_1}(t)$ be the characteristic function of $Y_1.$
       Then the characteristic function of $T_N^{(k)} $  satisfies, for $ t \in \mathbb{R}$,
\begin{align} \label{eqn337a}
     \mathbb E( e^{i t T_N^{(k)}}) = e^{ \dt_k (\phi_{Y_1}(t)-1 )} \to e^{ \dt (\phi_{Y_1}(t)-1 )} = \mathbb E( e^{i t T_N}), ~~ \text{as} ~ k \to \infty.
\end{align}
This shows that the sequence of CP rvs $T_N^{(k)} \stackrel{\mathcal L}{\longrightarrow}  T_N $, as $ k \to \infty.$ 
\end{remarks}

\noi Our focus will be on the CPD $T_N= C_{P}(\blmd),$ say. The mean and the variance of  $T_N$ follows
 easily. Since $E(N)= Var(N)$, we have 
\begin{align*}
	\E(T_N)= \mathbb E(Y_1) \mathbb E(N)=\jsy j {\lmd_j}
\end{align*} 
and similarly
\begin{align*}
	\V(T_N)= \mathbb E(Y_1^2) \mathbb E(N)= \jsy j^2 {\lmd_j}.
\end{align*}

\noi Next, we obtain the probability distribution of the CPD. Define  ${T_0 = 0 }~ a. s.$  For  $ 1 \leq m \leq n $, let
\begin{align} \label{3.4}
	\Dt_{m,n}=& \{(y_1, y_2,\ldots y_m): y_j \geq 1;  \sum_{j=1}^{m}y_{j}=n \},
\end{align} 
and $ \sum_{ y_j \in \Dt_{m,n}}   \lambda_{y_1} \cdots \lambda_{y_m }= 0$, if the set $\Dt_{m,n} $ is empty.
\begin{theorem} \em  (i): The PMF of the $C_P(\blmd)$, defined in  \eqref{neqn33}, is 
 \begin{align} \label{neqn38a}
\mathbb{P}\Big\{T_N=n\Big\}=
\begin{cases}
 e^{\dt}, \quad \text{if}~~n =0 \\
	e^{-\dt}\sum_{m=1}^{n}  \Big \{\sum_{ (y_1, y_2,\ldots y_m) \in \Dt_{m,n}}   \lambda_{y_1} \cdots \lambda_{y_m } \Big\} \frac{1 }{m!}, \quad \text{if}~~n \ge 1.
\end{cases}
\end{align}
(ii): An explicit expression in terms of Bell polynomials is
\begin{align} \label{neqn38b}
	\mathbb P(T_N= n )=& \frac{e^{-\dt}} {n!}   {B}_{n}\Big( u_{1}, u_{2}, \ldots,  u_{n}\Big), \quad n \ge 0,
\end{align}
where $ u_j= j! \lmd_j$, for $ 1 \le j  \le n.$
\end{theorem}

\begin{proof} (i):
First, clearly, $ \mathbb{P}\Big\{T_N=0\Big\}= \mathbb{P} (N=0)= e^{-\dt}$.
 For $n\ge 1$, we have 
\begin{align*} 
	\mathbb{P}\Big\{T_N=n\Big\}
	&= \sum_{m=1}^{n} \mathbb{P}\Big\{Y_{1}+Y_{2}+\cdots+Y_{m}=n\mid N=m\Big\} \mathbb{P}\Big\{N=m\Big\} \quad ( \because y_j \geq 1) \\
	&=\sum_{m=1}^{n} \mathbb{P}\Big\{Y_{1}+Y_{2}+\cdots+Y_{m}=n\Big\} \mathbb{P}\Big\{N=m\Big\}  \\
	&=\sum_{m=1}^{n} \Big \{\sum_{y_j \in \Dt_{m,n}}\mathbb{P}\Big\{Y_{1}=y_{1}, Y_{2}=y_{2}, \ldots, Y_{m}=y_{m}\Big\} \Big \}\mathbb{P}\Big\{N= m\Big\} \\
	&=\sum_{m=1}^{n}\sum_{y_j\in \Dt_{m,n}} \prod_{i=1}^{m} \mathbb{P}\Big\{Y_{i}=y_{i}\Big\}\mathbb{P}\Big\{N=m\Big\} \\
	&= \sum_{m=1}^{n} \Big (\sum_{y_j \in \Dt_{m,n}} \frac{\lambda_{y_1} \cdots \lambda_{y_m }} {\dt^ m} \Big )\frac{e^{-\dt}\dt ^m}{m!}\\
	&=e^{-\dt}\sum_{m=1}^{n}  \Big \{\sum_{ y_j \in \Dt_{m,n}}   \lambda_{y_1} \cdots \lambda_{y_m } \Big\} \frac{1 }{m!},\numberthis \label{pmfgpd}
\end{align*}
where  $\dt =\sum_{j=1}^{\infty}\lambda_j<\infty.$

(ii): \noi The probability generating function  of the $C_P(\blmd)$ (see \eqref{eqn33a} ) is
\begin{align*}
	H_{T_N}(z) =& \,e^{-\dt(1-H_{Y_1}(z))}\\
	= & \,e^{-\dt}\sum_{k=0}^{\infty} \dfrac{1}{k!} \Big( \sum_{j=1}^{\infty}\lambda_jz^j\Big)^k \\
	= & \, e^{-\dt}\sum_{k=0}^{\infty} \dfrac{1}{k!} \Big \{ \sum_{n=k}^{\infty} \hat{B}_{n, k}\Big(\lambda_{1}, \lambda_{2}, \ldots, \lambda_{n-k+1}\Big) z^{n} \Big \} \quad(\text{using}~ \eqref{221} )\\
	= & \, e^{-\dt}\sum_{n=0}^{\infty}  \Big \{ \sum_{k=0}^{n} \dfrac{1}{k!}\hat{B}_{n, k}\Big(\lambda_{1}, \lambda_{2}, \ldots, \lambda_{n-k+1}\Big) \Big \}  z^{n}\\
	= & \, e^{-\dt}\sum_{n=0}^{\infty}  \Big \{\dfrac{1}{n!} \sum_{k=0}^{n}  {B}_{n, k}\Big( 1!\lambda_{1}, 2!\lambda_{2}, \ldots, (n-k+1)!\lambda_{n-k+1}\Big) \Big \}  z^{n} ~~( \text {using \eqref{eqn312}})\\
	= & \, e^{-\dt}\sum_{n=0}^{\infty}  \Big \{\dfrac{1}{n!}   {B}_{n} \Big( 1!\lambda_{1}, 2!\lambda_{2}, \ldots, n!\lambda_{n} \Big) \Big \}  z^{n}
\end{align*}
which shows that, for $ n \geq 0,$
\begin{align*} 
	\mathbb P(T_N= n )= & \, \frac{e^{-\dt}} {n!}   {B}_{n}\Big( 1!\lambda_{1}, 2!\lambda_{2}, \ldots, n!\lambda_{n}\Big) \\
	= & \,  \frac{e^{-\dt}} {n!}   {B}_{n}\Big( u_{1}, u_{2}, \ldots,  u_{n}\Big), 
\end{align*}
which proves the result.

\end{proof}

\corollary \em   When  $\lambda_j=0,$ for $ j \geq 2$, we get $\dt= \lmd_1$ and $Y_j$'s are degenerate at 1. Hence, $\Dt_{m,n} $ is empty for $ 1 \leq m \leq n-1 $. If $m=n$, $\Dt_{n,n}=
(1, \ldots, 1) $,  then $ \sum_{ y_j \in \Dt_{n,n}}   \lambda_{y_1} \cdots \lambda_{y_1 }=  \lmd_1 \cdots \lmd_1= \lmd_1^n.$
Thus, \eqref{neqn38a} reduces to
$ \text{Poi}(\lmd_1),$  the distribution of $N_1$, as expected.\\

\noi (ii) Take now  $\lambda_j =0$ for $ j \geq k+1$, so that $ \dt= \dt_k$ in \eqref{neqn38a}, we get the probability distribution  of the Poisson distribution of order $k$ as
\begin{equation} \label{eqn34a}
	\mathbb P\{W_k=n\}=e^{-\dt_k} \sum_{m=1}^{n}  \Big \{\sum_{ y_j \in \Dt_{m,n}}   \lambda_{y_1} \cdots \lambda_{y_m } \Big\} \frac{1 }{m!},
\end{equation}
where $\Dt_{m,n}=\{(y_1,y_2,\ldots y_m): y_i\ge 1; \sum_{j=1}^{m}y_{j}=n\}$ and $\dt_k=\sum_{j=1}^{k}\lambda_j$, for all $\lambda_j>0$.

\vone

\remark  \em
\noi It is well known (see Comtet \cite{CT74}) that the complete Bell polynomials appear in the exponential of a formal power series:
\begin{align*} 
	\exp \Big (\sum _ {n = 1} ^  \infty \frac{a _ {n} x  ^ {n} }{n! } \Big)= 
	\sum _ {n = 0} ^  \infty   
	\frac{B_{n}(a_1, a_2, \ldots, a_{n}) }{n! } x^n. 
\end{align*}
\noi The above equation,  when $x=1$, leads to
\begin{align*} 
	\exp \Big (\sum _ {n = 1} ^  \infty \frac{a _ {n}  }{n! } \Big)= 
	\sum _ {n = 0} ^  \infty   
	\frac{B_{n}(a_1, a_2, \ldots, a_{n}) }{n! }. 
\end{align*}

\vspace{-.2cm}
If we take  $a_n= n! \lmd_n$, we get
\begin{align*} 
	\exp \Big (\sum _ {n = 1} ^ {\infty} \lmd_n   \Big)= \exp(\dt)=
	\sum _ {n = 0} ^  \infty   
	\frac{B_{n}(1!\lmd_1,  2!\lmd_2, \ldots, n! \lmd_n) }{n! }, 
\end{align*}
proving that \eqref{neqn38b} is indeed a valid probability distribution.

\remark  \em 
From \eqref{eqn226}  and \eqref{neqn38b},  
\begin{align} \label{eqn36}
    \mathbb P\{T_N= n \}&=  \frac{e^{-\dt}} {n!}   {B}_{n}\Big( 1!\lambda_{1}, 2!\lambda_{2}, \ldots, n!\lambda_{n}\Big) \nonumber \\ 
	&= {e^{-\dt}} \frac{\mathbb  P\{W_n=n\}}{ \mathbb  P(S_n=0)}  \nonumber \\
	 &= {e^{-(\dt-\dt_n)}} { \mathbb P\{W_n=n\}}, ~ n \in  \mathbb{Z}.
\end{align}

\noi Let now $\lmd_j =0$ for $ j \geq 2.$  Then $\dt= \dt_n=\lmd_1$ and $W_n=
X_1.$ Hence,
\begin{align*} 
\mathbb P\{T_N= n\}=  \mathbb P\{X_1= n \},
\end{align*}
as expected.

\noi Suppose now $\lmd_j =0$ for $ j \geq (k+1).$  Then $\dt= \dt_k$  and
 $W_n= W_k$  for $ n \geq (k+1).$ Hence,
\begin{align*} 
	\mathbb  P\{T_N= n\}= \begin{cases}
	{e^{-(\dt_k-\dt_n)}}  \mathbb  P\{W_n= n \}, & \text{if} ~ n \leq k \\
         \mathbb  P\{W_k= n\}, & \text{if} ~ n \geq k+1.
\end{cases}
\end{align*}
\noi For distribution of $W_k$, see  \eqref{eqn214} or  \eqref{eqn34a}.
\section{ A Compound Poisson Processes. }
Motivated by the properties of the $GPD$, we now extend it to  processes versions.

\begin{definition} \em (i):   {\bf A Generalized Poisson Process.} Let $\{N(t, \lmd_j) \}, j \geq 1,$ be a sequence of independent Poisson processes with
	parameters $\lmd_j$  satisfying $\dt= \jsy \lmd_j < \infty.$ We call the process
\begin{align}\label{4.1}
  G_P(t, \blmd) =\sum_{j=1}^{\infty} j N(t, {\lmd}_j), 
\end{align}
  the generalized Poisson process (GPP).\\ 
\noi (ii): {\bf Poisson Process of order $k$.} For $k \in \N$, let $\blmd^{(k)}= ( \lmd_1, \ldots, \lmd_k)$ and 
\begin{align} \label{4.2}
	G_P(t, \blmd^{(k)})=\sum_{j=1}^{k} j N(t, {\lmd}_j).
\end{align}
We call the generalized Poisson
process $\{ 	G_P(t, \blmd^{(k)}) \}_{t \ge 0}$ the Poisson process of order $k.$ 
This is the natural extension of Poisson distribution of order $k$ to Poisson process of order $k.$ For other extensions, see, for example, Sengar and Upadhye
\cite{AN23} and the references therein.
\end{definition}

\noi Note that the process $ G_P(t, \blmd)$ reduces to a $G_P(t, \blmd^{(k)})$, when $ \lmd_j= 0$ for $j \ge k+1.$
We next show that the GPP is indeed a compound Poisson process (CPP).
\begin{lemma} \em Let $Y_j$'s be i.i.d with distribution given in \eqref{neqn31} and 
${N(t, \dt)}$ be a Poisson process with the parameter $\dt=\jsy \lmd_j < \infty $ 
and is independent of the $Y_j$.
Then the GPP 
\begin{align}\label{43a}
	G_P(t, \blmd) \st \sum_{j=1}^{N(t, \dt)}Y_j, \quad t>0,
\end{align}
and hence is a compound Poisson process.
\end{lemma}

\begin{proof}
\noi For fixed $t>0$, the PGF of   $ G(t, \blmd)$ is 
\begin{align} \label{eqn44a}
	H_{G(t, \blmd)}(z)&= \mathbb{E}\Big(z^{\sum_{j=1}^{N(t, \dt)}Y_j}\Big) \non \\
           &=  \exp \Big( { - t \dt (1-F_{Y_1}(z)) } \Big) \non \\
           &= \exp \Big({ -t \dt \big(1- \sum_{j=1}^{\infty}z^{j}\frac{\lmd_{j}}{\dt}\big)} \Big)  \non \\
           &=  \exp \Big({ -t \dt + t \sum_{j=1}^{\infty}{\lmd_{j}}} z^{j}\Big).
\end{align}
Similarly, it follows from \eqref{neqn34} that
\begin{align} \label{eqn45a}
	\mathbb{E}\Big( z^{\sum_{j=1}^{\infty} j N(t, {\lmd}_j)}\Big) 
	=&  \exp \Big( \sum_{j=1}^{\infty} t \lmd_{j} (z^{j}-1) \Big) \nonumber \\
	=&  \exp \Big({ -t \dt + t \sum_{j=1}^{\infty}{\lmd_{j}}}z^{j} \Big). 	
\end{align}
The lemma follows from \eqref{eqn44a} and \eqref{eqn45a}.
\end{proof}

\noi Denote henceforth the compound Poisson process (CPP) $\sum_{j=1}^{N(t, \dt)}Y_j $ by $H(t, \blmd).$
\noi Taking $\lmd_j=0$ for $k\ge 1$,   the CPP becomes the Poisson process of order $k$, defined by
\begin{align} \label{4.2b}
	H(t, \blmd^{(k)}) \st \sum_{j=1}^{N(t, \dt_k)}Y_j,
\end{align}

\vspace{-.2cm}
where  $\blmd^{(k)}= ( \lmd_1, \ldots, \lmd_k)$ and $ \dt_k= \sum_{j=1}^{k} \lmd_j.$ 
\vspace*{.2cm}

\begin{remarks} \em  (i)
	\nin Since the Poisson process has independent increments, the CPP  $H(t, \blmd))$    has independent increments and also $ \mathscr{L}( H(s+t, \blmd)-H(s, \blmd))=  \mathscr{L}(H(t, \blmd))$. Let now $\phi(s)= E(e^{isY_1})$ be the characteristic function of $Y_1$. 
		Then, in view of \eqref{43a}, the characteristic function of $H(t, \blmd)$  is
	\begin{align*} 
		\mathbb  E(e^{i s H(t, \blmd)})=& \exp(\dt t (\phi(s)-1 )) \\
		=& \exp \Big(\dt t \int  (e^{isy}-1) d\mu(y)\Big), 
	\end{align*}
	which shows that  $H(t, \blmd)$ is infinitely divisible with the Levy measure $\nu= \dt t \mu$, where  $\mu$ is the distribution of $Y_1.$  Thus,  we have the the process $H(t, \blmd)$ is an integer-valued Levy process and such process have applications in financial
	econometrics (see Ole Bandorff-Nielsen et al. \cite {OBN12}).
   \nin Moreover, let  $A= \dt \int y d\mu$ and  $B= \dt \int y^2 d\mu$. Then it is
   well known that   $M_1(t)= H(t, \blmd)- A t$ and  $M_2(t)= (H(t, \blmd)- A t)^2 -B t$ are martingales with respect to the natural filtration. \\

\nin (ii) Let $g(s)= \mathbb  E(s^{Y_1})$ be the moment generating function of $Y_1$ and 
$ h(s)= 1-g(-s)$. Since $Y_j$'s are positive integer-valued rvs, we have by Theorem 3.5 of
\cite{OBN12} that  $ H(t, \blmd)= N(S(t), \dt)$ for some   subordinator $S(t) $  if and only if $ h(s)= \log \mathbb  E(e^{-sU})$ for some infinitely divisible rv $U >0.$
\end{remarks}

\noi The one-dimensional distributions of the process $H(t, \blmd)$ can be obtained from the CPD, by replacing  $\lmd_j$ by $t\lmd_j.$ 
This leads us to the following result.
\begin{theorem} \em
The one-dimensional distributions of the
	CPP are
	\begin{equation}\label{n43}
		\mathbb{P}\Big\{H(t, \blmd)=n \Big\}= \begin{cases}
		e^{-t\dt},  ~  \text{if} ~~ n=0, \\
		e^{-t\dt}\sum_{m=1}^{n}  \sum_{y_j \in \Dt_{m,n}} \prod_{i=1}^{m} \lambda_{y_i}  \frac{t^m }{m!}, ~~ \text{if}~~ n \geq 1, 
	      \end{cases}	
\end{equation}
where $\Dt_{m,n}$ is defined in \eqref{3.4}.\\
Alternatively, in terms of Bell polynomials,
\begin{align} \label{4.10}
	\mathbb{P}\Big\{H(t, \blmd)=n \Big\}= & \frac{e^{-t\dt}} {n!}   {B}_{n}\Big( u_1 t, u_2 t, \ldots, u_n t\Big), ~~ n \geq 0,
\end{align}
where $ u_j= j! \lmd_j, 1 \le j \le n.$
\end{theorem}

\corollary \em Taking $\lambda_j =0$ for $ j \geq k+1$ in \eqref{n43}, we get the probability mass function (PMF) of the Poisson process of order $k$ as
\begin{equation}\label{3.10}
	\mathbb{P}\Big\{H(t, \blmd^{(k)})=n \Big\} =  \begin{cases}
	e^{-t\dt_k},  ~~ \text{if}~~ n=0, \\
	e^{-\delta_k }\sum_{m=1}^{n}\sum_{(y_1, y_2,\ldots y_m) \in\Dt_{m,n}}  \prod_{i=1}^{m} \lambda_{y_i} \frac{ t^m}{m!}, ~~ \text{if}~~ n \geq 1,
  \end{cases}
\end{equation}


\subsection{Mean, Variance and Covariance Functions}
The mean, variance and covariance functions of  $G(t, \blmd)$ are as follows. For positive reals $s$ and $t$,
\begin{align}
	\mathbb{E}\Big(H(t, \blmd)\Big)&=\mathbb{E}(Y_1)\mathbb{E}\Big(N(t, \dt)\Big) = t \sum_{j=1}^{\infty}j\lambda_{j}; \label{n410} \\
	\V(H(t, \blmd))&= \mathbb{E}(Y_1^2)\mathbb{E}\Big(N(t, \dt)\Big)=t \sum_{j=1}^{\infty}j^2\lambda_{j}, \label{n411}
\end{align}
since $ \mathbb{E}\Big(N(t, \dt)\Big)= \V(N(t, \dt)). $ 

\nin Using  $\operatorname{Cov}(N(s, \dt), N(t, \dt))= \min(s, t) \dt$,  we have for $ 0< s \le t$,  
\begin{align}
	\operatorname{Cov}(H(s, \blmd), H(t, \blmd))&= \V (Y_1) \mathbb{E}\Big(N(s, \dt)\Big) + \mathbb{E}^2(Y_1) \mathbb{E} \Big(\CV(N(s, \dt), N(t, \dt) )  \Big) \non \\
	  &= \V (Y_1) s \dt + \mathbb{E}^2(Y_1)  s \dt =s \dt \mathbb{E}(Y_1^2) \label{n412}  \\ 
	 &= s \sum_{j=1}^{\infty}j^2\lambda_{j} \non
\end{align}

 \nin The CPP exhibits an over-dispersion property. That is, for $t>0$,  
\begin{align*}
	\operatorname{Var}(H(t, \blmd))-\mathbb{E}\Big( H(t, \blmd)\Big)= &t \sum_{j=1}^{\infty}j^2\lambda_{j}- t \sum_{j=1}^{\infty}j\lambda_{j} \\ 
	=& t \jsy j(j-1) \lmd_j>0.
\end{align*}

\nin The following definition of long-range property (LRD) and short-range property (SRD) property will be used (see, for e.g., Maheshwari and Vellaisamy \cite{MV16}):

\begin{definition} \em
	Let $s>0$ be fixed and $\{X(t)\}_{t\ge0}$ be a stochastic process such that the correlation function
	\begin{equation*}
		\operatorname{Corr}(X(s),X(t))\sim c(s)t^{-\gamma},\ \ \text{as}\ t\to\infty,
	\end{equation*}
	for some $c(s)>0$. Then, the process $\{X(t)\}_{t\ge0}$ is said to have the LRD property if $\gamma\in(0,1)$ and the SRD property if $\gamma\in (1,2)$.	
\end{definition}

\begin{lemma} \em The CPP $H(t, \blmd)$ exhibits the LRD property.
\end{lemma}
	\begin{proof} From the equations \eqref{n410}-\eqref{n412}, the correlation function of  the CPP is, for $0<s< t,$
	\begin{align*}
	\operatorname{Corr}\Big(H(s, \blmd), H(t, \blmd)\Big) 
	 =& \frac{s\dt \mathbb{E}(Y_1^2)} {\sqrt{s\dt \mathbb{E}(Y_1^2)} \sqrt{t \dt \mathbb{E}(Y_1^2)} } \\
	=& \sqrt{s}t^{-1/2},
\end{align*}
and so it has the LRD property.
\end{proof}	

\begin{lemma}\label{lem41} \em
	The following asymptotic result holds for the CPP:
	\begin{equation}\label{limit}
		\lim_{t\to\infty}\frac{H(t,\blmd)}{t}= \dt  \mathbb{E}(Y_1)~~ a.s.
	\end{equation}	
\end{lemma}
\begin{proof} By the renewal theorem of the Poisson process,
	\begin{equation*}
	\lim_{t\to\infty}\frac{N(t, \lmd_j)}{t}=\lambda_{j}, ~ a.s.
	\end{equation*}	
    Hence,
    \begin{align*}
    	\lim_{t\to\infty}\frac{H(t,\blmd)}{t}{=} &\sum_{j=1}^{\infty}j\lim_{t\to\infty}\frac{N(t, \lmd_j)}{t} \\
    	=&\sum_{j=1}^{\infty}j\lambda_{j}~~ a.s \\
        = & \, \dt  \mathbb{E}(Y_1) ~~ a.s.
\end{align*}
\end{proof}

\noi The above result implies also the convergence in law, that is,
\begin{equation} \label{n47}
\frac{H(t,\blmd)}{t}~ {\stackrel{\mathscr{L}}{\to}}~ \dt  \mathbb{E}(Y_1), ~~\text{as}~ t \to \infty.
\end{equation}	

\nin Also,  replacing $\dt$ by $\dt_k$, we obtain the corresponding result for the Poisson process of order $k.$

 \section{A  Time Fractional Compound Poisson Process }
\nin  In the past few decades,  fractional Poisson processes
and their extensions have received  considerable attention of  several researchers,  see, for instance, Laskin \cite{Lask09}, Begin and Vellaisamy \cite{BegVel18} and Kataria and Khandakar \cite{KKMK21}, Gara {\it et al.} \cite{GOS17} and the references therein.
A multivariate extension of the fractional Poisson process
is discussed in Begin and Macci \cite{BegMac16}.

\nin First, we briefly recall some properties of the time fractional Poisson process, which will be used later.
\subsection{Time Fractional Poisson process }\label{secfpp}
Let $0<\beta<1$ and $\lmd \geq 0.$ Unlike the  Poisson process, the time fractional Poisson process (TFPP) has neither independent nor stationary increments. Also, it is not a Markov process. \\
Let $\{U_i\}_{i=1}^{\infty}$ be a sequence of i.i.d. positive rvs, denoting
the inter-arrival times of an event, with the common cumulative distribution function (CDF) 
\begin{equation}
	\label{mlcdf}
	F_U (u) = \mathbb{P}\{U \leq u\} = 1 - M_\beta (- \lmd u^\beta),
\end{equation}
where $M_\beta (-t^\beta)$ is the one-parameter Mittag-Leffler function. That is, $U_i$'s follow
the Mittag-Leffler distribution with density (see \eqref{mld2})
$$f_{U_1}(u)=\lambda u^{\beta-1}M_{\beta,\beta}(-\lambda u^\beta),~u\geq 0.$$
The sequence of the {\em epochs}, denoted by $\{V_n\}_{n=1}^\infty$, is given by the sums of the inter-arrival
times
\begin{equation}
	\label{epoch}
	V_n = \sum_{j=1}^n U_j, ~ n \ge 1.
\end{equation}
The epochs represent the times in which events arrive or occur. The PDF of the $n$-th epoch $V_n,$  the $n$-fold convolution of $U_j$'s,  is given by (see \eqref{mld3})
\begin{equation}
	\label{epochpdf}
	f_{V_n} (v) = f^{*n}_U (v) = \beta \lmd^n \frac{v^{n\beta -1}}{(n-1)!} M^{(n)}_\beta (- \lmd v^\beta).
\end{equation}
The counting process $N_{\beta}(t, \lmd)$ that counts the number of epochs (events) up to time $t$, assuming that
$U_0 = 0$, is called the TFPP. Note $N_{\beta}(t, \lmd)$ is given by
\begin{equation}
	\label{countingprocess}
	N_{\beta}(t, \lmd) = \max \{ n:\; V_n \leq t \}.
\end{equation}

\noindent The {PMF} $p_{_{\beta}}(n|t,\lambda)$ of the TFPP is given by (see Laskin \cite{Lask09} or Meerschaert {\it et al.} \cite{MNV11}) 
\begin{align*}
	p_{_{\beta}}(n|t,\lambda)
	=& 	{(\lambda t^{\beta})^n} M_{\bt, n\bt+1}^{n+1} (-\lmd t^\bt) \non \\  
	 	=& 	\frac{(\lambda t^{\beta})^n}{n!} M_{\bt}^{(n)} (-\lmd t^\bt),~
	 	\numberthis
	\label{n513} 
\end{align*}
using \eqref{2.5}. Also,  $p_{_{\beta}}(n|0,\lambda)=1\text{ if }n=0 \text{ and is zero if }n\geq1.$

\noi Also, the PMF $p_{_{\beta}}(n|t,\lambda)$ satisfies (see Begin and Orsingher \cite{BegOrs09})
 \begin{equation} \label{neqn56}
 	\partial_t^{\beta} p_{_{\beta}}(n|t,\lambda)= \begin{cases}
 		-\lambda p_{_{\beta}}(0|t,\lambda),  \quad \text{if}~~  n=0, \\
 		-\lambda\Big[ p_{_{\beta}}(n|t,\lambda)-p_{_{\beta}}(n-1|t,\lambda)\Big],
 		\quad \text{if}~~ n \ge 1,
 	\end{cases}
 \end{equation}
 where $\partial_t^{\beta}$ denotes the Caputo fractional derivative defined in \eqref{cd}.

\noi The mean and the variance of the TFPP are given by (see  \cite{Lask03})
 \begin{align}
 	\E \Big( N_{\beta}(t,\lambda)\Big)  &= qt^{\beta}, \quad q=\lambda/\Gamma(1+\beta); \label{n57} \\
 	\V \Big( N_{\beta}(t,\lambda)\Big)  &=qt^{\beta}\Big[1+qt^{\beta}\Big(\frac{\beta B(\beta, 1/2)}{2^{2\beta-1}}-1\Big)\Big],\label{fppvar}
 \end{align}
 where   $B(a,b)$ denotes the beta function. An alternative form for Var[$N_{\beta}(t,\lambda)$] is given in \cite[eq. (2.8)]{BegOrs09} as
 \begin{align}\label{n59}
 	\V \Big( N_{\beta}(t,\lambda)\Big) 
 	= & q t^{\beta}  + \lmd^2 t^{2\bt} Q(\bt),
 \end{align}
where 
\begin{equation} \label{n510}
Q {(\bt) }=\frac{1}{\beta}\Big(\frac{1}{\Gamma(2\beta)}-\frac{1}{\beta\Gamma^{2}(\beta)}\Big).
\end{equation}

\noindent Also,  the covariance functions (see Leonenko {\it et al.} \cite{LMS14}) of the TFPP is given by 
\begin{align} \label{fppcov}
	\CV \Big( N_{\beta}(s,\lambda),N_{\beta}(t,\lambda)\Big)
	 &= qs^{\beta} + \lmd^2 \CV(E_{\bt}(s),  E_{\bt}(s)) \\
	 &=qs^{\beta}+ ds^{2\beta}+ q^{2}[\beta t^{2\beta}B(\beta,1+\beta;s/t)-(st)^{\beta}], \non
\end{align}
\noindent
$0<s\leq t$, where   $d=\beta q^{2}B(\beta, 1+\beta)$, and $B(a,b;x)=\int_{0}^{x}t^{a-1}(1-t)^{b-1}dt,~0<x<1$, is the incomplete beta function.

 \noindent It is also known that (see Meerschaert {\it et al.} \cite{MNV11}) when $0<\beta<1,$
 \begin{equation}\label{n511}
 	N_{\beta}(t,\lambda)  \st N(E_{\beta}(t),\lambda), 
 \end{equation}
 where $\{E_{\beta}(t)\}_{t\geq0}$ is the inverse $\beta$-stable subordinator and is independent of $\{N(t,\lambda)\}_{t\geq0}$.\\
 The PGF of the  TFPP is given by 
 \begin{align} \label{n512}
 	H_{N_\beta (t, \lmd)}(z) = \E (z^{	N_{\beta}(t,\lambda)})= M_\beta(-\lambda t^\beta (1-z)),~|z| \leq 1,
 	\end{align}
 see for example  Maheshwari and Vellaisamy\cite {MV19}. 

\subsection{Time Fractional Compound Poisson Process. }
 The study of compound fractional Poisson processes has been of recent interest; see, for example, Begin and Macci \cite{BegMac22} and Kataria and Khandakar \cite{KKMK22}. In this section, we introduce a new time-fractional compound Poisson process, motivated by the compound Poisson process discussed in Section 4.

\begin{definition}  \em  (i) Let $\{N_{\beta}(t, \lmd_j) \}, j \geq 1,$ be a sequence of independent time fractional  Poisson processes with
	parameters  $\beta$ and $\blmd= \{\lmd_j\}_{j \ge 1}$.   Let $Y_j$'s be i.i.d with distribution given in \eqref{neqn31} and is independent of the process 
	$\{N_{\beta}(t, \dt) \}$, where $\dt= \jsy \lmd_j.$ We call the process
	\begin{align} \label{n514}
		G_{\beta}(t, \blmd)=\sum_{j=1}^{\infty} j N_{\beta}(t, {\lmd}_j),
	\end{align}
a	 time fractional generalized Poisson process (TFGPP).
\end{definition}

\nin The next result shows that $G_{\beta}(t, \blmd)$ is indeed a  time fractional  compound Poisson processes,  and is also a subordinated GPP.
	\begin{proposition} \em
	Let  $G(t, \blmd)$ and  $G_{\beta}(t, \blmd)$  be  the GPP  and TFGPP respectively. 
	For fixed $t>0$,
	\begin{equation} \label{n515}
		G_{\beta}(t, \blmd)\st \sum_{j=1}^{N_{\beta}(t, \dt)}Y_j \st G_{P}(E_{\bt}(t), \dt),
	\end{equation}
	where $\{E_{\beta}(t)\}_{t\geq0}$ is an independent inverse $\beta$-stable subordinator. 
\end{proposition}

\begin{proof}
\noi Note,  for fixed  $t > 0$,
\begin{align*} 
	G_{\beta}(t, \blmd)= & \sum_{j=1}^{\infty} j N_{\beta}(t, {\lmd}_j) \\
	\st & \sum_{j=1}^{\infty} j N(E_{\beta}(t), {\lmd}_j) ~~ (\text{using \eqref{n511}})\\
	\st & \sum_{j=1}^{N(E_{\beta}(t), {\dt})}Y_j ~~(\text{using \eqref{43a}})\\
    	\st & \sum_{j=1}^{N_{\beta}(t, \dt)}Y_j ~~(\text{using \eqref{n511}}).
\end{align*}
\noi From the above results and \eqref{neqn33}, we have
\begin{align} \label{n516}
	G_{\beta}(t, \blmd) \st \sum_{j=1}^{N(E_{\beta}(t), {\dt})}Y_j
 \st  G_{P}(E_{\bt}(t), \dt),
\end{align}
showing that it is a subordinated GPP.
\end{proof}

\nin Henceforth, we call the process $\sum_{j=1}^{N_{\beta}(t, \dt)}Y_j $ the time fractional compound Poisson process (TFCPP) and denote it by $H_{\beta}(t, \blmd).$ Similarly,  the time fractional compound Poisson process
of order $k$, denoted by $H_{\beta}(t, \blmd^{(k)})$, can be defined by replacing $\dt$ by $\dt_k$ in the result given in \eqref{n516}.

\vspace{.3cm}
\nin The PGF of the TFCPP is
\begin{align}\label{pgf}
	K_{H_{\beta}(t, \blmd)}(z)&=\mathbb{E}\Big( z^{H_{\beta}(t, \blmd)}\Big)\nonumber\\
	 &= K_{N_{\beta}(t, {\dt})} (K_{Y_1}(z))\nonumber\\
	&=M_{\beta}\Big(-\dt t^\beta\Big(1-K_{Y_1}(z) \Big)\Big)  ~~ (\text{using \eqref{n512}})  \nonumber\\
	&=M_{\beta}\Big(-t^{\beta} \sum_{j=1}^{\infty}\lambda_{j}
	   \Big(1-z^{j}\Big)\Big),
\end{align}
where $M_{\beta}$ is defined in \eqref{mlf22}.

\subsection{Mean, variance and Covariance Functions of TFCPP}
Let $H_{\beta}(t, \blmd)$ be the TFCPP defined in \eqref{n515}. Then  mean, variance and covariance functions of of $H_{\beta }(t, \blmd)$ for $s\le t$ are as follows.
Let $q=\dt /\Gamma(1+\beta)$. Then, from \eqref{n57} and \eqref{n59},
\begin{align}\label{519a} 
	\mathbb{E}\Big( H_{\beta}(t, \blmd)\Big)&=  \mathbb{E}(Y_1)\mathbb{E} \big(N_{\beta}(t, \dt)  \big)=qt^\beta \mathbb{E}(Y_1);	 
\end{align}
and
\begin{align}\label{519b}
\operatorname{Var}(H_{\beta}(t, \blmd))	&=\operatorname{Var} (Y_1) \mathbb{E} \big(N_{\beta}(t, \dt) +\mathbb{E}^2(Y_1) \operatorname{Var} (N_{\bt}(t, \dt)) \non  \\
    &=\operatorname{Var} (Y_1) \mathbb{E} \big(N_{\beta}(t, \dt) +\mathbb{E}^2(Y_1) \big[\mathbb{E} \big(N_{\beta}(t, \dt)\big) + \dt^2 t^{2\bt} Q(\bt) \big] \non \\
    &=\mathbb{E}(Y_1^2)\mathbb{E} \big(N_{\beta}(t, \dt)\big) + \mathbb{E}^2(Y_1) \dt^2 t^{2\bt} Q(\bt) \non  \\
    &=\mathbb{E}(Y_1^2)qt^\beta  + \mathbb{E}^2(Y_1)\dt^2 t^{2\bt} Q(\bt), 
\end{align}
where $Q(\bt) $ is defined in \eqref{n510}.

\noi Also, for $0<s \le t$,
\begin{align} \label{n522}
	\operatorname{Cov}(H_{\beta}(s, \blmd),H_{\bt}(t, \blmd))&=\operatorname{Var} (Y_1) qs^{\bt}+\mathbb{E}^2(Y_1) \operatorname{Cov} \big(N_{\bt}(s, \dt), N_{\bt}(t, \dt)\big),
\end{align}
where $\dt =\displaystyle\sum_{j=1}^{\infty}\lambda_{j}$. Note that $\operatorname{Var}(N_{\bt}(t, \dt))$ and $\operatorname{Cov}(N_{\bt}(s, \dt), N_{\bt}(t, \dt))$ are given in \eqref{fppvar} and \eqref{fppcov}, respectively.

\remark \em The TFCPP exhibits the over-dispersion property. For $t>0$,
\begin{align*}
	\operatorname{Var}(H_{\beta}(t, \blmd))-\mathbb{E}\Big( H_{\beta}(t, \blmd)\Big)
	&=qt^{\bt}[\mathbb{E}(Y_1^2)- \E(Y_1)]+\mathbb{E}^2(Y_1)\dt^2 t^{2\bt} Q(\bt) \\
	&=\frac{qt^{\bt}}{\dt} \jsy j(j-1) \lmd_j+\mathbb{E}^2(Y_1)\dt^2 t^{2\bt} Q(\bt)\\ 
    &>0,
\end{align*}
since $ Q(\bt) >0$ for $ 0 < \bt \in (0,1)$ (see Vellaisamy and Maheshwari \cite{PVAM18}, p. 87.).

\vspace{.3cm}

\begin{theorem} \em For $ 1 \le k \leq n$, let $$\Dt_{k,n}= \{(y_1, y_2,\ldots y_k): y_j \geq 1;  \sum_{j=1}^{k}y_{j}=n \}. $$
(i) The $PMF$ of  the TFCPP $H_{\beta}(t, \blmd))$ is
		\begin{align}
	\mathbb{P}\Big\{H_{\beta}(t, \blmd)=n\Big\}= \begin{cases}
	M_{\beta}(-\delta t^\beta), ~~\text{if}~~ n=0 \\
		\frac{1}{n!} \sum_{k=0}^{n}\Big\{ \sum_{ y_j \in \Dt_{k,n}} \prod_{i=1}^{k} \lambda_{y_i} \Big\} t^{k\beta} M_{\beta}^{(k)}\Big(-\dt t^{\beta}\Big),~~ \text{if}~~ n \ge 1	
     \end{cases}
	\end{align}	

(ii) Alternatively, in terms of Bell polynomials, 
\begin{equation} \label{n522a}
	\mathbb{P}\Big\{H_{\beta}(t, \blmd)=n\Big\}=\frac{1}{n!} \mathbb{E} \Big({e^{-\dt E_{\beta}(t)}}    {B}_{n}\big( u_1 E_{\beta}(t), u_2 E_{\beta}(t), \ldots, u_n E_{\beta}(t) \big) \Big)
\end{equation}
for $n\ge 0$ and $u_j= j! \lmd_j, 1 \le j \le n.$
\end{theorem}
\begin{proof}
	(i) From \eqref{n513}, we have $\mathbb{P}\Big\{H_{\beta}(t, \blmd)=0\Big\}= \mathbb{P}\Big\{N_{\beta}(t,\lambda)=0\Big\}=M_{\beta}(-\delta t^\beta). $
	 For  $n\ge 1$, 	
	\begin{align*}
		\mathbb{P}\Big\{H_{\beta}(t, \blmd)=n\Big\}&=\sum_{k=1}^{n} \mathbb{P}\Big\{Y_{1}+Y_{2}+\cdots+Y_{k}=n\Big\} \mathbb{P}\Big\{N_{\beta}(t,\dt)=k\Big\}  \quad ( \because y_j \geq 1)\\ 
		&=\sum_{k=1}^{n}\sum_{\Dt_{k,n}}\mathbb{P}\Big\{Y_{1}=y_{1}, Y_{2}=y_{2}, \ldots, Y_{k}=y_{k}\Big\}\mathbb{P}\Big\{N_{\beta}(t,\dt)=k\Big\} \\
		&=\sum_{k=1}^{n}\sum_{\Dt_{k,n}} \prod_{i=1}^{k} \mathbb{P}\Big\{Y_{i}=y_{i}\Big\}\mathbb{P}\Big\{N_{\beta}(t,\dt)=k\Big\} \\
		&= \frac{1}{n!} \sum_{k=1}^{n}\Big\{ \sum_{ y_j \in \Dt_{k,n}}  \frac{\lambda_{y_1} \cdots \lmd_{y_k}}{\dt^k}\Big\}(\dt t^{\beta})^k M_{\beta}^{(k)}\Big(-\dt t^{\beta}\Big) ~~(\text{using}~\eqref{n513} )\\
		&=  \frac{1}{n!} \sum_{k=1}^{n}\Big\{ \sum_{ y_j \in \Dt_{k,n}} {\lambda_{y_1} \cdots \lmd_{y_k}}\Big\} t^{\beta k} M_{\beta}^{(k)}\Big(-\dt t^{\beta}\Big).
	\end{align*}
(ii) Note that,  for fixed $t$, and $ n \geq 0$,
\begin{align*}
	\mathbb{P}\Big\{H_{\beta}(t, \blmd)=n\Big\}&= \mathbb{P}\Big(  \sum_{j=1}^{N(E_{\beta}(t), ~{\dt})}Y_j=n \Big)  \\
    &= \mathbb{E}\Big(\mathbb{P}\Big(\sum_{j=1}^{N(E_{\beta}(t){\dt})}Y_j=n 
    \mid  E_{\beta}(t)\Big)\Big)   \\
    &= \frac{1}{n!} \mathbb{E} \Big({e^{-\dt E_{\beta}(t)}}    {B}_{n}\big(u_1 E_{\beta}(t), u_2 E_{\beta}(t), \ldots, u_n E_{\beta}(t) \big) \Big),	
\end{align*}
using \eqref{4.10}. The density of $ E_{\beta}(t)$ is given in \eqref{instab-den}
\end{proof}

\begin{remark} \em
	Let the rv $W_\beta$ denote the waiting time of the first TFCPP event.  
	Then, the distribution of $W_\beta$ satisfies 
	\[
	\mathbb P(W_\beta > t) = \mathbb P(H_\beta(t; \lmd) = 0) = M_\beta(-\lambda t^\beta), \quad t > 0,
	\]  
	which coincides with the first waiting time of TFPP (see Beghin and Orsingher \cite{BegOrs09}).  Also, it is known that 
	that the TFPP is a renewal process (see Meerschaert {\it et al}. \cite{MNV11}).  
	Since one-dimensional distributions of TFPP and TFCPP are different,  the TFCPP  is not a renewal process.
	
\end{remark}

\begin{theorem} \em
	The one-dimensional distributions of $H_{\beta}(t, \blmd)$ are not infinitely divisible.
\end{theorem}
\begin{proof} Using the well-known result (see \cite{MS13}) that $ E_{\beta}(t) \st t^{\beta}
	E_{\beta}(1)$, we obtain
	\begin{align*}
   \frac{H_{\beta}(t, \blmd)}{t^{\beta}} & \st  \frac{G(E_{\beta}(t), \blmd)}{t^{\beta}} ~~(\text{use \eqref{n515}}) \\
   & \st \frac{G(t^{\bt}E_{\beta}(1), \blmd)}{t^{\beta}}   \\
   & \st E_{\beta}(1) \frac{G(t^{\bt}E_{\beta}(1), \blmd)}{t^{\beta}
   	 E_{\beta}(1)} \\
  & \stackrel{\mathscr{L}}{\rightarrow} E_{\beta}(1) \dt \mathbb{E} (Y_1),
\end{align*}
as $t \to \infty$, using \eqref{limit}.	

\noi Suppose $ H_{\beta}(t, \blmd), t>0, $ is infinitely divisible (i.d.). Then this implies
$\frac{H_{\beta}(t, \blmd)}{t^{\beta}}$	is i.d. for every $t>0$ and also its limit
$ E_{\beta}(1) \dt \mathbb{E} (Y_1)$ or equivalently $E_{\beta}(1) $ is  also i.d.
(see Sato \cite{ST99}). But this is a contradiction, since $E_{\beta}(t) $ is not i.d. for any $t>0$
(see Steutel and Van Harn \cite{SVH04}).	
\end{proof}

\subsection{Moments and factorial moments of TFCPP}
 For real-valued functions $f$ and $g$, let $f^{(k)}$ denote its $k$-th derivative and $g(f)$ denote the composite function. The following two results are from Johnson \cite{John02}, equation (3.3) and equation (3.6).

\noi \textbf{Hoppe's formula}. If $g$ and $f$ are functions with a sufficient number of derivatives, then
\begin{equation}\label{eqn61}
   (g(f) )^{(m)}=\sum_{k=0}^{m} \frac{g^{(k)}(f)}{k !} A_{k, m}(f),
\end{equation} 
where  $ A_{0, 0}=1$, $ A_{0, m}=0$ for $m \ge 1$  and  
\begin{equation} \label{eqn62}
A_{k, m}(f)=\sum_{j=0}^{k}\binom{k}{j}(-f)^{k-j} (f^{j})^{(m)}, \quad 1 \le k \le m.
\end{equation}

\noi The next lemma is from \cite[eq. (3.6)]{John02}.
\begin{lemma}\label{jl}  \em
	(i): If $f_{1}, f_{2}, \ldots, f_{k}$ are functions with a sufficient number of derivatives, then
\begin{equation} \label{eqn63}
	\Big(f_{1} f_{2} \cdots f_{k}\Big)^{(m)}=\sum_{j_{1}+\cdots+j_{k}=m}\binom{m}{j_{1}, \ldots, j_{k}} f_{1}^{\left(j_{1}\right)} \cdots f_{k}^{\left(j_{k}\right)}.
\end{equation}
 (ii): When $f_i=f, 1 \le i \le k$, 
	\begin{equation} \label{eqn64}
	(f^k)^{(m)}=\sum_{j_{1}+\cdots+j_{k}=m}\binom{m}{j_{1}, \ldots, j_{k}} f^{\left(j_{1}\right)} \cdots f^{\left(j_{k}\right)}.
\end{equation}
\end{lemma}

\noi Note that the moment generating function $K_{\beta}(s, t)$, $s\ge 0,$ of TFCPP is (see
\eqref{pgf})
\begin{equation}\label{mgf}
	K_{\beta}(s, t)=\mathbb{E}\Big(e^{-s H_{\beta}(t, \blmd)}\Big)=M_{\beta }\Big(t^{\beta} \sum_{j=1}^{\infty}\lambda_{j}\Big(e^{-s j}-1\Big)\Big). 
\end{equation}

\begin{theorem} \em  Let $Y_j$'s be i.i.d. rvs with distribution given in \eqref{neqn32} and $T_k= \sum_{j=1}^{k} Y_j$,
	for $ k \ge 1.$
	Then the $r$-th raw moment of the TFCPP is given by
	\begin{align*}
		\mathbb{E}\Big( H_{\beta}^{r}(t, \blmd)\Big)
        =& 	 \sum_{k=1}^{r} \frac{t^{k \beta} \dt^k}{\Gamma(k \beta+1)} \E(T_k^r).
		\end{align*}
\end{theorem}
\begin{proof}
	Using Hoppe's formula  in \eqref{eqn61}, we get
	\begin{align*}
			\mathbb{E}\Big( H_{\beta}^{r}(t, \blmd)\Big) &=(-1)^r\Big.\frac{\partial^{r} K_{\beta}(s, t)}{\partial s^{r}}\Big|_{s=0} \\
		&=\Big.\sum_{k=0}^{r} \frac{(-1)^r}{k !} M_{\beta, 1}^{(k)}\Big(t^{\beta} \sum_{j=1}^{\infty}\lambda_{j}\Big(e^{-s j}-1\Big) \Big) A_{k, r}\Big(t^{\beta} \sum_{j=1}^{\infty}\lambda_{j}\Big(e^{-s j}-1\Big)\Big)\Big|_{s=0}\numberthis\label{em1},
	\end{align*}
	where $A_{k, r}$, for $0 \le k \le r$, is defined in  \eqref{eqn62}.\\
  \noi   From \eqref{2.5}, we obtain 
  \begin{align*}
  	\Big. M_{\beta}^{(k)}\Big(t^{\beta} \sum_{j=1}^{\infty}\lambda_{j}\Big(e^{-s j}-1\Big)\Big)\Big|_{s=0} &=\Big.k ! M_{\beta, k \beta+1}^{k+1}\Big(t^{\beta} \sum_{j=1}^{\infty}\lambda_{j}\Big(e^{-s j}-1\Big)\Big)\Big|_{s=0} \\
  	&=\frac{k !}{\Gamma(k \beta+1)},\numberthis\label{em2}
  \end{align*}  
using the fact that  $ M_{\al, \bt}^{\ga}(0)={1}/{\Gamma(\beta)}.$\\
\noi Next, for $ 0 \le k \le r$,
  	\begin{align*}
		&\Big. A_{k, r}\Big(t^{\beta} \sum_{j=1}^{\infty}\lambda_{j}\Big(e^{-s j}-1\Big)\Big)\Big|_{s=0} \\
		&\quad=\Big.\sum_{m=0}^{k}\binom{k}{m}\Big(-t^{\beta} \sum_{j=1}^{\infty}\lambda_{j}\Big(e^{-s j}-1\Big)\Big)^{k-m} \frac{d^{r}}{ds^{r}}\Big(t^{\beta} \sum_{j=1}^{\infty}\lambda_{j}\Big(e^{-s j}-1\Big)\Big)^{m}\Big|_{s=0} \\
		&\quad=\Big.t^{k \beta} \frac{d^{r}}{ds^{r}}\Big(\sum_{j=1}^{\infty}\lambda_{j}\Big(e^{-s j}-1\Big)\Big)^{k}\Big|_{s=0}\numberthis\label{em3},
	\end{align*}
The last step follows since only the case corresponding to $m=k$ remains in its previous step and also by using $\sum\limits_{j=1}^{\infty}\lambda_{j}\Big(z^{j}-1\Big)|_{z=1}=0$.

\noi 	Now, by Lemma \ref{jl}, we get
	\begin{align*}
		\Big.\frac{d^{r}}{ds^{r}}\Big(\sum_{j=1}^{\infty}\lambda_{j}\Big(e^{-s j}-1\Big)\Big)^{k}\Big|_{s=0} &=(-1)^r  \sum_{ n_1+\ldots+ n_{k}=r} \binom{r}{n_1, \ldots, n_k}\prod_{i=1}^{k}  \frac{d^{n_{i}}}{d s^{n_{i}}}\Big(\sum_{j=1}^{\infty}\lambda_{j}\Big(e^{-s j}-1\Big)\Big)\Biggr|_{\rho=0} \\
		&=(-1)^r   \sum_{ n_1+\ldots + n_{k}=r} \binom{r}{n_1, \ldots, n_k} \prod_{i=1}^{k}  \sum_{j=1}^{\infty}\lambda_{j}j^{n_{i}} \\
     &=(-1)^r  \dt^k  \sum_{ n_1+\ldots + n_{k}=r} \binom{r}{n_1, \ldots, n_k} \prod_{i=1}^{k}
        \E(Y_i^{n_i})\\
            &=(-1)^r \dt^k \E(Y_1+\ldots+Y_k)^{r}. \numberthis\label{em4}
	\end{align*}
Substituting \eqref{em2}-\eqref{em4} in \eqref{em1}, we get the result.
\end{proof}

\noi For a random variable $Z$, let $Z^{(r)}= Z(Z-1)\cdots(Z_r+1)$ so that $\E(Z^{(r)})$ denote its $r$-th factorial moment.

\begin{theorem}  \em
	The $r$-th factorial moment $\mathcal{M}_{r}^{\beta}(t) $ of the TFCPP is given by
	\begin{equation*}
\E(H^{(r)}_{\beta}(t, \blmd))	= \sum_{k=1}^{r} \frac{t^{k \beta}\dt^k}{\Gamma(k \beta+1)}
	\sum_{ n_1+\ldots, n_{k}=r} \binom{r}{n_1, \ldots + n_k}  \prod_{i=1}^{k} 
	\E(Y^{(n_i)}).	
	\end{equation*}
\end{theorem}

\begin{proof}
From \eqref{pgf} and using Hoppe's formula, we get
\begin{align*}
\E(H^{(r)}_{\beta}(t, \blmd))	 &=\Big.\frac{\partial^{r} H_{H_{\beta}(t, \blmd)}(z)}{\partial z^{r}}\Big|_{z=1} \\
	&=\Big.\frac{\partial^{r}}{\partial z^{r}}\Big\lbrace M_{\beta}\Big(t^{\beta} \sum_{j=1}^{\infty}\lambda_{j}\Big(z^{j}-1\Big)\Big)\Big\rbrace \Big|_{z=1} \\
	&=\Big.\sum_{k=0}^{r} \frac{1}{k !} M_{\beta}^{(k)}\Big(t^{\beta} \sum_{j=1}^{\infty}\lambda_{j}\Big(z^{j}-1\Big) \Big) A_{k, r}\Big(t^{\beta} \sum_{j=1}^{\infty}\lambda_{j}\Big(z^{j}-1\Big)\Big)\Big|_{z=1}\numberthis\label{fm1},
\end{align*}
where, for  $ 0 \le k \le r$,
\begin{align*}
	&\Big.A_{k, r}\Big(t^{\beta} \sum_{j=1}^{\infty}\lambda_{j}\Big(z^{j}-1\Big)\Big)\Big|_{z=1} \\
	&\quad=\Big.\sum_{m=0}^{k}\binom{k}{m}\Big(-t^{\beta} \sum_{j=1}^{\infty}\lambda_{j}\Big(z^{j}-1\Big)\Big)^{k-m} \frac{d^{r}}{dz^{r}}\Big(t^{\beta} \sum_{j=1}^{\infty}\lambda_{j}\Big(z^{j}-1\Big)\Big)^{m}\Big|_{z=1} \\
	&\quad=\Big.t^{k \beta} \frac{d^{r}}{dz^{r}}\Big(\sum_{j=1}^{\infty}\lambda_{j}\Big(z^{j}-1\Big)\Big)^{k}\Big|_{z=1}\numberthis\label{fm2}\\
    &=t^{k \beta}\sum_{ n_1+\ldots+ n_{k}=r} \binom{r}{n_1, \ldots, n_k} \prod_{i=1}^{k} \frac{d^{n_{i}}}{d z^{n_{i}}}\Big(\sum_{j=1}^{\infty}\lambda_{j}\Big(z^{j}-1\Big)\Big)\Biggr|_{z=1} \\
    &=t^{k \beta} \sum_{ n_1+\ldots + n_{k}=r} \binom{r}{n_1, \ldots, n_k}  \prod_{i=1}^{k}  \sum_{j=1}^{\infty}\lambda_{j}j(j-1)\cdots (j-n_i+1)\\ 
    &= t^{k \beta} \dt^k \sum_{ n_1+\ldots + n_{k}=r} \binom{r}{n_1, \ldots, n_k}  \prod_{i=1}^{k} 
    \E(Y^{(n_i)}). \numberthis\label{fm3}
\end{align*}

Also, as seen  in \eqref{em2},
\begin{align*}
	\Big.M_{\beta}^{(k)}\Big(t^{\beta} \sum_{j=1}^{\infty}\lambda_{j}\Big(z^{j}-1\Big)\Big)\Big|_{z=1}
	&=\frac{k !}{\Gamma(k \beta+1)}\numberthis\label{fm4}.
\end{align*}
The result now follows by using \eqref{fm2}-\eqref{fm4} in \eqref{fm1}. 
\end{proof}

\subsection{Long Range Dependence Properties}
Here, we  discuss the long-range property (LRD)  of the TFCPP. An important property of the FPP is that it, unlike the classical Poisson process, has the LRD property
as shown in Biard and Saussereau \cite{BDSS14}. This makes the FPP  quite useful
in many applications that arise in finance, econometrics and hydrology.
The LRD property of the 
fractional negative binomial process is studied by Maheshwari and Vellaisamy 
\cite{MV16}.

\vspace{.3cm}
\nin The mean and variance functions  of the  inverse $\beta$-stable subordinator are given by (see Leonenko et al., \cite{LMS14}, Equation (8) and Equation (11))
\begin{equation*}
\mathbb{E}\Big(E_{\beta}(t)\Big)=\frac{t^{\beta}}{\Gamma(\beta+1)}
\end{equation*}
and
\begin{equation*}
\operatorname{Var}\Big(E_{\beta}(t)\Big)=t^{2 \beta} Q(\bt),
\end{equation*}
where $Q(\bt)$ is defined in \eqref{n510}.

\noi Also, for $0<s \leq t$, we have from \eqref{n522} and \eqref{fppcov},
\begin{align*}
\operatorname{Cov}(H_{\beta}(s, \blmd),H_{\beta}(t, \blmd))&=\V (Y_1) qs^{\bt}+\mathbb{E}^2(Y_1) \CV(N_{\bt}(s, \dt), N_{\bt}(t, \dt))\\
&=\V (Y_1) qs^{\bt}+\mathbb{E}^2(Y_1)\Big( qs^\beta+\dt^2 \operatorname{Cov}\Big(E_{\beta}(s), E_{\beta}(t)\Big)\Big)\\
&=qs^{\bt}\mathbb{E}(Y_1^2)+\dt^2\mathbb{E}^2(Y_1)\operatorname{Cov}\Big(E_{\beta}(s), E_{\beta}(t)\Big) \numberthis\label{cov}.
\end{align*}
For large $t$, it is known that (see Leonenko et al. \cite{LMS14})
\begin{equation}\label{cov1}
\operatorname{Cov}\Big(E_{\beta}(s), E_{\beta}(t)\Big) \sim \frac{s^{2 \beta}}{\Gamma(2 \beta+1)}.
\end{equation}
Thus, form  \eqref{cov1} in \eqref{cov}, we get
\begin{equation} \label{eqn73}
	\operatorname{Cov}(H_{\beta}(s, \blmd),H_{\beta}(t, \blmd))\sim \E(Y_1^2) + 
	\dt^2\mathbb{E}^2(Y_1) \frac{s^{2 \beta}}{\Gamma(2 \beta+1)},	 
\end{equation}
as $t\rightarrow\infty$.

\noi Using the above result, we prove the following result.

\begin{theorem} \em
	The TFCPP $ H_{\beta}(t, \blmd)$ possesses the LRD property.
\end{theorem}
\begin{proof}
	For fixed $s>0$ and large $t$, we have, from \eqref{eqn73}  and \eqref{519b},
\begin{align*}
	\operatorname{Corr}\big(H_{\beta}(s, \blmd),~&H_{\beta}(t, \blmd)\big)\\
	&=\frac{\operatorname{Cov}\Big(H_{\beta}(s, \blmd),H_{\beta}(t, \blmd)\Big)}{\sqrt{\operatorname{Var}\Big(H_{\beta}(s, \blmd)\Big)} \sqrt{\operatorname{Var}\Big(H_{\beta}(t, \blmd)\Big)}}\\
    & \sim \frac{\E(Y_1^2) + 
    	\dt^2\mathbb{E}^2(Y_1) \frac{s^{2 \beta}}{\Gamma(2 \beta+1)}}
        { \Big\{ H_{\beta}(s, \blmd) \Big[ \mathbb{E}(Y_1^2)qt^\beta  +\dt^2 \mathbb{E}^2(Y_1) t^{2\bt} Q(\bt)  \Big]  \Big\}^{1/2} } \\
	&\sim Z(s)t^{-\beta}, ~~\text{for ~ large}~ t,
\end{align*}
where 
\begin{equation*}
	Z(s)=  \frac{\E(Y_1^2) + 
		\dt^2\mathbb{E}^2(Y_1) \frac{s^{2 \beta}}{\Gamma(2 \beta+1)}}
	{ \Big\{\dt^2 H_{\beta}(s, \blmd)  \mathbb{E}^2(Y_1)  Q(\bt)   \Big\}^{1/2} }.
\end{equation*}
Since $0<\beta<1$, the  TFCPP  possesses the LRD property.
\end{proof}

\begin{remark} \em
For a fixed $h > 0$, the increment of TFCPP is defined as  
\[
H^h_\beta(t; \tau) := H_\beta(t +h, \lmd) - H_\beta(t; \lmd), \quad t \geq 0.
\]  
It can be shown that the increment process $\{H^h_\beta(t, \beta)\}_{t \geq 0}$ exhibits the SRD property.  
The proof is similar  to that of Theorem 1 of Maheshwari and Vellaisamy \cite{MV16} and  hence is omitted.
\end{remark}

\vspace{.3cm}
\noi Next, we obtain the fractional Kolmogorov forward type equations for the $PMF$ of  $H_{\beta}(t, \blmd).$
Let
\begin{align*}
	\mathbb{P}\Big\{H_{\beta}(t, \blmd)=n\Big\}&=q(n|t, \blmd), \quad n \ge 0, \\	
    \mathbb{P}(T_k=n) &= h_k(n), \quad n \ge 1,
\end{align*}
be respectively the PMF of the TFCPP and $T_k= \sum_{j=1}^{k} Y_j,  k \ge 1.$

\begin{theorem} \em Let $0 < \bt \le 1$  and  $\partial_t^{\bt}$ be the  Caputo-fractional derivative defined in \eqref{cd}. The $PMF$  $q(n|t, \lmd)$ of the TFCPP  satisfies the following fractional differential
	equation:
	 \begin{equation} \label{neqn540}
		\partial_t^{\beta} q(n|t,\blmd)= \begin{cases}
			-\dt p_{_{\beta}}(0|t,\blmd),  \quad \text{if}~~  n=0, \\
			-\dt q(n|t,\blmd) + \dt \sum_{k=1}^{n} h_k(n) p_{\bt}(k-1|t, \dt), 
			\quad \text{if}~~ n \ge 1,
		\end{cases}
	\end{equation}
where $p_{\bt}(n|t, \lmd) $ denotes the $PMF$ of the fractional Poisson process.
\end{theorem}

\begin{proof}  First note that 
\begin{equation*} 
q(0|t,\blmd)= \mathbb{P}\Big\{H_{\beta}(t, \blmd)=0\Big\}=p_{\bt}(0|t, \dt)	
\end{equation*}
and so 
\begin{equation*} 
\partial_t^{\beta}	q(0|t,\blmd)= \partial_t^{\beta}p_{\bt}(0|t, \dt)= - \dt p_{\bt}(0|t, \dt),	
\end{equation*}
which follows form \eqref{neqn56}.\\
When $n \ge 1$,
\begin{align*}
q(n|t,\blmd)= & \sum_{k=1}^{n} \mathbb{P}(S_k=n) \mathbb{P}(N_{\bt}(t, \dt)=k  ) \\	
          = & \sum_{k=1}^{n} h_k(n) p_{\bt}(k|t, \dt).
\end{align*}
This implies, using \eqref{neqn56},
\begin{align*}
	\partial_t^{\beta}q(n|t,\blmd)
	= & \sum_{k=1}^{n} h_k(n) \partial_t^{\beta} p_{\bt}(k|t, \dt) \\
    = & -\dt \sum_{k=1}^{n} h_k(n)\big[p_{\bt}(k|t, \dt)-p_{\bt}(k-1|t, \dt)\big] \quad 
    (\text{using \eqref{neqn56}})\\
    = & -\dt q(n|t,\blmd) + \dt \sum_{k=1}^{n} h_k(n)p_{\bt}(k-1|t, \dt), 
\end{align*}
where $p_{\bt}(n|0, \dt)= 1 ~\text{if}~ n=0$ and is zero ~\text{if}~ $  n \ge 1.$ This proves the result.
\end{proof}

\nin Similar results for $H_{\beta}(t, \blmd^{(k)})$ can be obtained by replacing $\dt$ by $\dt_k.$

\subsection{Hitting Times of $ H_{\beta}(t, \blmd)$} The first hitting times of $ H_{\beta}(t, \blmd)$ if defined by
\begin{equation} \label{n541}
	T_k^{\beta} = \inf \{ t \ge 0:  H_{\beta}(t, \blmd) \ge k\}.
\end{equation}
Then, from \eqref{n522},
\begin{align*}
  \mathbb{P}(T_k^{\beta}<s)=& \, \mathbb{P}(H_{\beta}(s, \blmd) \ge k) \\
                 =& \sum_{m=k}^{\infty}\mathbb{P}(H_{\beta}(s, \blmd)=m) \\
                 =&\sum_{m=k}^{\infty} \frac{1}{m!} \mathbb{E} \Big({e^{-\dt E_{\beta}(s)}}    {B}_{n}\big(u_1 E_{\beta}(s), u_2 E_{\beta}(s), \ldots, u_n E_{\beta}(s) \big) \Big)\\
                 =& \, \ell(s), ~ (\text{say}).
\end{align*}
Then the density of $T_k^{\beta}$ of $ f_{T_k^{\beta}}(s)= \ell^{'}(s).$

\subsection{Application to Risk Modeling} The TFCPP $H_{\beta}(t, \blmd)$ can serve
an important stochastic model for several disciplines. Note that integer-valued Levy process have been used in financial
econometrics (see Ole Bandorff-Nielsen et al. \cite {OBN12}). 
 As the TFCPP has the long range dependence, over-dispersion property, flexibility of the claim size distribution, it can be a more realistic and useful model for financial and insurance modeling and  risk theory.\\ 
 Following Beghin and Macci \cite{BegMac14}, consider the risk model  
\begin{equation} \label{n542}
	R_{\beta}(t)= \, ct - \sum_{j=1}^{N_{\beta}(t, \dt)} Y_j=ct-H_{\beta}(t, \blmd),\\  
\end{equation}
where $c$ denotes the constant premium rate. Here $ N_{\beta}(t, \dt) $ denotes the number of claims at time $t$ and $Y_j$'s represent discrete claim sizes which are assumed to be independent of $ N_{\beta}(t, \dt). $   

\nin The expected value of $ R_{\beta}(t) $ is
\begin{align*}
	\mathbb{E}(R_{\beta}(t))=& \, ct- \mathbb E( H_{\beta}(t, \blmd)) \\
	                   = & \, ct - q t^{\beta} \mathbb{E}(Y_1),
\end{align*}
using \eqref{519a} and  from \eqref{519b}
\begin{align*}
	\operatorname{Var}(R_{\beta}(t))=& \operatorname{Var}( H_{\beta}(t, \blmd))=
	\mathbb{E}(Y_1^2)qt^\beta  + \mathbb{E}^2(Y_1)\lmd^2 t^{2\bt} Q(\bt).
\end{align*}
Note that the claim-size distribution of $Y_1$ is given in \eqref{neqn32} and so $\mathbb{E}(Y_1)$ and $\mathbb{E}(Y_1^2)$ can be computed easily.

\begin{remark} \em Let $D_{\beta}(t)$ be the stable subordinator and  $ {\tilde N}_{\beta}(t, \dt)=N(D_{\bt}(t), \dt)  $ be the space fractional
	Poisson process. 
	One could study, similar to the TFCPP discussed in this paper, the space fractional compound
	Poisson process  (see, for example, Orsingher and Polito \cite{OrsPoli12}, Begin and Vellaisamy \cite{BegVel18}) defined by
	\begin{equation} \label{n540}
		S_{\beta}(t, \blmd)\st \sum_{j=1}^{\tilde{N}_{\beta}(t, \dt)}Y_j \st G_p(D_{\bt}(t), \dt),
	\end{equation}
	by subordinating the generalized Poisson process $G_P(t, \blmd)$ to the stable subordinator $D_{\beta}(t, \blmd)$ and derive the corresponding results. In addition, the study of time-changed versions of such subordinated
	processes also could be of interest.
\end{remark}

\end{document}